\documentclass[dvipsnames,11pt]{amsart}
\usepackage[margin=1.35in]{geometry}
\usepackage[OT2,T1]{fontenc}
\usepackage[utf8]{inputenc}
\usepackage[dvipsnames]{xcolor}
\usepackage{tikz-cd}
\usetikzlibrary{decorations.pathmorphing}
\usepackage{stmaryrd}
\SetSymbolFont{stmry}{bold}{U}{stmry}{m}{n}
\usepackage[pagebackref]{hyperref}
\hypersetup{
    colorlinks=true,
    linkcolor=blue,
    filecolor=magenta,  
    urlcolor=cyan,
    citecolor=ForestGreen,
    pdfpagemode=FullScreen,
    }
\hypersetup{pdfpagemode=FullScreen,colorlinks=true,citecolor=ForestGreen}
\usepackage[alphabetic]{amsrefs}
\usepackage{microtype,amsmath,amsfonts,amssymb,amsthm,mathrsfs,bm}
\usepackage[scr=rsfs,cal=cm]{mathalfa}
\usepackage[]{spectralsequences}

\usepackage{bbm}
\usepackage{dsfont}
\usepackage{mathptmx}   
\usepackage{helvet}     

\DeclareRobustCommand{\coprod}{\mathop{\text{\fakecoprod}}}
\newcommand{\fakecoprod}{%
  \sbox0{$\prod$}%
  \smash{\raisebox{\dimexpr.9625\depth-\dp0}{\scalebox{1}[-1]{$\prod$}}}%
  \vphantom{$\prod$}%
}

\theoremstyle{plain}
\newtheorem{theorem}[]{Theorem}
\newtheorem{proposition}[theorem]{Proposition}
\newtheorem{lemma}[theorem]{Lemma}

\newtheorem{corollary}[theorem]{Corollary}

\theoremstyle{definition}
\newtheorem{definition}[]{Definition}
\newtheorem{instance}[definition]{Example}
\newtheorem{situation}[subsection]{}
\newtheorem{remark}[definition]{Remark}
\newtheorem{assumption}{Hypothesis}

\newtheorem*{notation}{Notation}

\usepackage{enumitem}
\numberwithin{theorem}{section}
\numberwithin{definition}{section}
\numberwithin{equation}{section}
\setlist[itemize]{leftmargin=2em}

\author{Tsung-Ju Lee}
\email{tsungju@gs.ncku.edu.tw}
\address{Department of Mathematics\\ 
National Cheng Kung University\\
No.~1, Daxue Rd.\\ 
East District, Tainan 70101, Taiwan.}

\title{Finite distance problem on the moduli of
non-K\"{a}hler Calabi--Yau \texorpdfstring{\(\partial\bar{\partial}\)}{ddbar}-threefolds}

\date{\today}

\begin{document}
\begin{abstract}
In this article, we study the finite distance problem 
with respect to the period-map metric on the moduli of
non-K\"{a}hler Calabi--Yau \(\partial\bar{\partial}\)-threefolds via Hodge theory.
We extended C.-L.~Wang's finite distance criterion for one-parameter degenerations
to the present setting. As a byproduct, we also obtained a
sufficient condition for a non-K\"{a}hler Calabi--Yau to support 
the \(\partial\bar{\partial}\)-lemma which generalizes
the results by Friedman and Li. We also proved that the non-K\"{a}hler Calabi--Yau
threefolds constructed by Hashimoto and Sano support 
the \(\partial\bar{\partial}\)-lemma.
\end{abstract}
\maketitle
\tableofcontents

\section{Introduction}{}
\subsection{Motivations}
Let \(Y\) be a three-dimensional Calabi--Yau (CY) manifold
(for instance a quintic threefold in \(\mathbf{P}^{4}\))
and \(C_{1},\ldots,C_{r}\) be disjoint smooth rational curves
in \(Y\) whose normal bundle \(\mathcal{N}_{C_{i}\slash Y}\) is
isomorphic to
\(\mathscr{O}(-1)\oplus\mathscr{O}(-1)\).
One can construct
a contraction \(\pi\colon Y\to\bar{X}\) where \(\bar{X}\) is a
singular threefold having
\(r\) ordinary double points (ODPs) \(p_{1},\ldots,p_{r}\)
whose pre-images under \(\pi\) are \(C_{1},\ldots,C_{r}\).
If one further assumes that \([C_{1}],\ldots,[C_{r}]\) span \(\mathrm{H}_{2}(X;\mathbb{C})\)
and there are \(m_{1},\ldots,m_{r}\in\mathbb{Q}\) such that
\begin{equation}
\sum_{i=1}^{r} m_{i} [C_{i}] = 0 \in\mathrm{H}_{2}(X;\mathbb{C})~\mbox{with}~m_{i}\ne 0~\mbox{for all}~i,
\end{equation}
the result of Friedman 
\cite{1986-Friedman-simultaneous-resolution-of-threefold-double-points} implies that
\(\bar{X}\) is smoothable.
Let \(X\) be a smoothing. 
We thus obtain 
a \emph{conifold transition} \(X\nearrow Y\); 
a complex degeneration \(X\rightsquigarrow\bar{X}\) followed by a resolution \(Y\to \bar{X}\). 
One checks that
\(b_{2}(X)=0\) and in particular, \(X\) is non-K\"{a}hler. 
The construction was firstly described by H.~Clemens around 1985. 
Based on this, M.~Reid speculated that there could be a 
single irreducible moduli space of possibly non-K\"{a}hler CY threefolds
such that any CY threefold can be connected to 
another member in this moduli through conifold transitions;
this is also known as Reid's fantasy \cite{1987-Reid-the-moduli-space-of-3-folds-with-k-0-may-nevertheless-be-irreducible}.
Here, by a non-K\"{a}hler CY manifold we mean a compact complex manifold
with trivial canonical bundle.
From this perspective, one inevitably 
bumps into non-K\"{a}hler CY manifolds and it 
is of importance to investigate their moduli spaces.

In the K\"{a}hler regime, 
Yau's theorem \cite{1978-Yau-on-the-ricci-curvature-of-a-compact-kahler-manifold-and-the-complex-monge-ampere-equation-i}, among other things, produces
a canonical metric on the moduli space of \emph{polarized} CY manifolds, 
called the \emph{Weil--Petersson} (WP) metric.
To explore the structure of moduli spaces, one may 
study the WP geometry on them.
It is known that the WP metric is incomplete in general
as Candelas \textit{et al}.~found
a nodal CY threefold having finite WP distance
\cite{1990-Candelas-Green-Hubsch-rolling-among-calabi-yau-vacua}. 
Later, C.-L.~Wang gave a Hodge-theoretic criterion for finite distance fibers
in the case of one-parameter degenerations 
\cites{1997-Wang-on-the-incompleteness-of-the-weil-petersson-metric-along-degenerations-of-calabi-yau-manifolds,2003-Wang-quasi-hodge-metrics-and-canonical-singularities}
and proposed a finite distance conjecture for higher-dimensional bases.
The author subsequently verified Wang's conjecture up to a codimension
two subset in the moduli space
\cite{2018-Lee-a-Hodge-theoretic-criterion-for-finite-WP-degenerations-over-a-higher-dimensional-base}. 

Due to the absence of K\"{a}hler metrics,  
the standard tools such as Hodge theory do not apply
when we study the moduli spaces of non-K\"{a}hler CY manifolds.
Nevertheless, R.~Friedman \cite{2019-Friedman-the-ddbar-lemma-for-general-clemens-manifolds}
and C.~Li \cite{2022-Li-polarized-hodge-structures-for-clemens-manifolds}
showed that the smoothing \(X\) described in the previous paragraph
supports the \(\partial\bar{\partial}\)-lemma,
which allows us to define Hodge decomposition 
on their cohomology groups. 
Besides, it is proven by J.~Fu, J.~Li, and S.-T.~Yau 
that such an \(X\) carries a
balanced metric \cite{2012-Fu-Li-Yau-balanced-metrics-on-non-kahler-calabi-yau-threefolds}. 
Recall that a hermitian metric on a compact
complex manifold of complex dimension \(n\) is called \emph{a balanced metric} if its 
metric two form \(\omega\)
satisfying the equation \(\mathrm{d}\omega^{n-1}=0\).
Based on Fu--Li--Yau's 
balanced structure, T.~Collins, S.~Picard,~and S.-T.~Yau 
showed that the holomorphic tangent bundle of \(X\)
admits a Hermitian--Yang--Mills connection \cite{2024-Collins-Picard-Yau-stability-of-the-tangent-bundle-through-conifold-transitions}.
T.~Collins, S.~Gukov, S.~Picard, and S.-T.~Yau also
constructed special Lagrangian spheres in \(X\)
\cite{2023-Collins-Gukov-Picard-Yau-special-lagrangian-cycles-and-calabi-yau-transitions}.

Non-K\"{a}hler CY \(\partial\bar{\partial}\)-manifolds have been extensively studied
in the literature. Let \(W\) be a CY \(\partial\bar{\partial}\)-manifold
of complex dimension \(n\).
D.~Popovici proved that both the
unobstructedness theorem and local Torelli theorem hold
for \(W\)
\cite{2019-Popovici-holomorphic-deformations-of-balanced-calabi-yau-d-dbar-manifolds}.
Moreover, if the second Hodge--Riemann bilinear relation holds on 
\(\mathrm{H}^{n}(W;\mathbb{C})=\oplus_{p+q=n}\mathrm{H}^{p,q}(W)\),
we can then pullback the corresponding 
Fubini--Study metric on \(\mathbf{P}\mathrm{H}^{n}(W;\mathbb{C})\)
to the moduli space \(S\)
via the local Torelli immersion \(S\to \mathbf{P}\mathrm{H}^{n}(W;\mathbb{C})\).
Note that
we do not need any metrics to define the pairing on the middle cohomology group.
In this manner, we obtain the so-called \emph{period-map metric} on \(S\)
(cf.~\cite{2019-Popovici-holomorphic-deformations-of-balanced-calabi-yau-d-dbar-manifolds}). 
Taking balanced metrics into account, D.~Popovici also introduced several variants of Weil--Petersson metrics on the
moduli of compact CY \(\partial\bar{\partial}\)-manifolds
and studied the relationship among them
\cite{2019-Popovici-holomorphic-deformations-of-balanced-calabi-yau-d-dbar-manifolds}.

\subsection{Statements of the main results}
The aim of this note is to study the finite distance problem 
with respect to the period-map metric for
one-parameter degenerations.

Let \(f\colon\mathcal{X}\to\Delta\) be a one-parameter degeneration of
\(\partial\bar{\partial}\)-manifolds of complex dimension \(n\).
We may assume that \(f\) is a semi-stable family by a finite base change and blow-ups;
these operations do not affect the finite distance property.
Put \(E:=f^{-1}(0)\) and let \(E_{1},\ldots,E_{m}\) be irreducible components of \(E\).
Suppose that any intersection of \(E_{i}\)'s supports the \(\partial\bar{\partial}\)-lemma.
In which case, 
following Steenbrink \cite{1975-Steenbrink-limits-of-hodge-structures}, one can define 
the limiting 
mixed Hodge structure \((\mathcal{F}^{\bullet}_{\mathrm{lim}},\mathcal{W}(M)_{\bullet},
\mathrm{H}^{n}(X;\mathbb{C}))\) where \(X\) is a general fiber. Let \(T\) be
the monodromy operator on \(\mathrm{H}^{n}(X;\mathbb{C})\) 
for \(f\) and \(N:=\log T\) be the nilpotent operator.
Then \(N\) is of type \((-1,-1)\) with respect to the limiting mixed Hodge structure.
Under another mild assumption
\begin{assumption}
\label{assumption:isomorphism}
The nilpotent operator \(N\) induces an isomorphism
between relevant quotients of
the monodromy weight filtration, namely
\begin{eqnarray*}
N^{k}\colon 
\mathrm{Gr}^{\mathcal{W}(M)}_{n+k}\mathrm{H}^{n}(X;\mathbb{C})\to 
\mathrm{Gr}^{\mathcal{W}(M)}_{n-k}\mathrm{H}^{n}(X;\mathbb{C})
\end{eqnarray*}
is an \emph{isomorphism} for each \(0\le k\le n\).
\end{assumption}
\noindent we will prove
\begin{theorem}[= Theorem \ref{thm_infinite-distance}]
Assume Hypothesis \ref{assumption:isomorphism}. Then
\(0\in\Delta\) has infinite distance with respect to the period-map metric
if \(N\mathcal{F}^{n}_{\mathrm{lim}}\ne 0\).
\end{theorem}
Note that the statement above indeed consists of two parts.
First, since 
we do not know whether or not the second Hodge--Riemann bilinear relation holds 
on \(\mathrm{H}^{n-1,1}(X)\), 
we need to show that the period-map metric is an honest metric. 
Fortunately, this can be done when 
\(N\mathcal{F}^{n}_{\mathrm{lim}}\ne 0\).
After resolving this issue, we then compute the distance and
prove that \(0\in\Delta\) has infinite
distance following C.-L.~Wang's approach
\cite{1997-Wang-on-the-incompleteness-of-the-weil-petersson-metric-along-degenerations-of-calabi-yau-manifolds}.

The other case \(N\mathcal{F}^{n}_{\mathrm{lim}}=0\) is somewhat subtler. 
Consider a more stringent assumption.
\begin{assumption}
\label{assumption:com-kahler}
There exists a cohomology class in 
\(\mathrm{H}^{2}(E;\mathbb{R})\) whose restriction to \(E_{i}\)
of \(E\) provides a K\"{a}hler class for \(E_{i}\).
\end{assumption}
This tacitly asserts that each component of \(E\) is K\"{a}hler and
in particular, it implies that any arbitrary intersections of \(E_{i}\)'s are K\"{a}hler and
hence the \(\partial\bar{\partial}\)-lemma holds on them.

Recall that a CY 
\(\partial\bar{\partial}\)-manifold \(X\) is \emph{strict},
if \(\mathrm{H}^{d}(X;\mathcal{O}_{X})=0\)
for \(0<d<n\). We then have the following result
for \(N\mathcal{F}^{n}_{\mathrm{lim}}=0\).
\begin{theorem}[= Proposition \ref{prop:finite-distance-polarized}]
\label{thm:main-2}
Under Hypothesis \ref{assumption:com-kahler},
if the general fiber \(X\) of \(f\) is a strict smooth CY 
\(\partial\bar{\partial}\)-threefold and \(N\mathcal{F}^{3}_{\mathrm{lim}}=0\), then
the second Hodge--Riemann bilinear relation holds
on \(\mathrm{H}^{2,1}(X)\). Consequently, the period-map metric
is truly a metric.
\end{theorem}

To prove the theorem, we shall follow C.~Li's
approach in \cite{2022-Li-polarized-hodge-structures-for-clemens-manifolds}.
Let us consider Deligne's splitting \(I^{p,q}\) for 
Steenbrink's limiting mixed Hodge structure 
on \(\mathrm{H}^{3}(X;\mathbb{C})\).
Firstly, we construct a basis for each \(I^{p,q}\) and hence
obtain a basis for \(\mathcal{F}_{\mathrm{lim}}^{2}\). 
Secondly, we extend the proceeding basis to become a local frame 
for Deligne's canonical extension. Finally we untwist the local frame 
to get a (multi-valued) frame for \(\mathcal{F}^{2}_{t}\)
where \(\mathcal{F}^{2}\) is the holomorphic bundle over \(\Delta^{\ast}\)
whose fiber over \(t\) is the vector space
\(\mathrm{H}^{3,0}(\mathcal{X}_{t})\oplus \mathrm{H}^{2,1}(\mathcal{X}_{t})\).
We then check the positivity using the multi-valued frame.

\begin{corollary}[= Theorem \ref{thm:finite-distance}]
In the situation of Theorem \ref{thm:main-2}, \(0\in\Delta\) has finite distance
with respect to the period-map metric.
\end{corollary}

As an application, we can use the multi-valued frame to obtain
the following proposition, which slightly generalizes the results of R.~Friedman
\cite{2019-Friedman-the-ddbar-lemma-for-general-clemens-manifolds} and
C.~Li \cite{2022-Li-polarized-hodge-structures-for-clemens-manifolds}.
\begin{proposition}[= Corollary \ref{cor:finite-distance-ddbar}]
Let \(g\colon\mathcal{X}\to\Delta\) be a semi-stable degeneration.
We assume that
the general fiber \(X\) of \(g\) has 
complex dimension \(3\) and satisfies the following conditions.
\begin{itemize}
\item The central fiber of \(f\) is at a finite distance with respect to
the perid-map metric;
\item \(\mathrm{H}^{i}(X;\mathcal{O}_{X})=\mathrm{H}^{0}(X;\Omega_{X}^{j})=0\)
for \(1\le i,j\le 2\).
\end{itemize}
Then under Hypothesis \ref{assumption:isomorphism}, the \(\partial\bar{\partial}\)-lemma holds on \(X\).
\end{proposition}

Hashimoto and Sano constructed smooth non-K\"{a}hler CY threefolds with arbitrarily large
\(b_{2}\) \cite{2023-Hashimoto-Sano-examples-of-non-kahler-calabi-yau-3-folds-with-arbitrarily-large-b2}. 
We may also apply the techniques here to show that
\begin{proposition}[= Theorem \ref{thm:hs-example-ddbar} and Theorem \ref{thm:hs-example-polarized}]
Let \(X\) be the smooth non-K\"{a}hler CY threefold constructed in 
\cite{2023-Hashimoto-Sano-examples-of-non-kahler-calabi-yau-3-folds-with-arbitrarily-large-b2}.
Then \(X\) is a \(\partial\bar{\partial}\)-manifold. Moreover, the second Hodge--Riemann
bilinear relation holds on \(\mathrm{H}^{2,1}(X)\).
\end{proposition}

\begin{remark}
Some results presented in this manuscript were also obtained by K.-W.~Chen independently
\cite{2024-Chen}, where the Hodge structure on the smoothings is carefully studied.
\end{remark}

\subsection*{Acknowledgment}
The author would like to thank Professor Chin-Lung Wang for his interest as well as
many valuable conversations with the author on this project. 
He would like to thank Professor Shing-Tung Yau for his interest and comments on 
the manuscript. The author 
also wants to thank 
Chung-Ming Pan for inspiring discussions. Part of the results presented here
was announced in the workshop ``East Asian Symplectic Conference 2023''
in October, 2023. In a communication with C.-L.~Wang, the author was informed that K.-W.~Chen
has also obtained similar results in this direction independently. The author would also like
to thank him for suggestions to improve the presentation.

\section{Preliminaries}
\subsection*{Notation}
Let \(X\) be a compact complex manifold. 
For a sheaf \(\mathscr{F}\) on \(X\), the
notation \(\mathrm{H}^{k}(X;\mathscr{F})\) stands for 
the usual sheaf cohomology on \(X\).
For a bounded complex of \(\mathbb{C}\)-sheaves \(\mathscr{F}^{\bullet}\) on \(X\),
we denote by \(\mathbf{H}^{k}(X;\mathscr{F}^{\bullet})\)
the \(k\)\textsuperscript{th} hypercohomology group.

\subsection{The \texorpdfstring{\(\partial\bar{\partial}\)}{}-lemma}
\label{subsection:ddbar-lemma}
We begin with the statement of \(\partial\bar{\partial}\)-lemma.
Let \(X\) be a compact complex manifold and 
\(\Omega_{X}^{p}\) be the sheaf of holomorphic \(p\)-forms on \(X\).
There is a complex of sheaves \(\Omega_{X}^{\bullet}\)
\begin{equation}
\label{equation:resolution-holomorphic}
\begin{tikzcd}[column sep=1em]
&0\ar[r]&\Omega_{X}^{0}\ar[r,"\partial"]&\Omega_{X}^{1}\ar[r,"\partial"]&\cdots
\ar[r,"\partial"]&\Omega_{X}^{n}\ar[r]&0,~n=\dim X,
\end{tikzcd}
\end{equation}
which is quasi-isomorphic to the constant sheaf \(\mathbb{C}_{X}\) on \(X\). 
The hypercohomology of \(\Omega_{X}^{\bullet}\) computes the 
usual cohomology (singular cohomology) of \(X\). 

The truncated complexes 
\begin{equation*}
\mathrm{F}^{p}\Omega_{X}^{\bullet}:= [0\to\cdots \to 0\to \Omega_{X}^{p}
\to\cdots\to\Omega_{X}^{n}\to 0]
\end{equation*}
with \(\Omega_{X}^{n}\) at the \(n\)\textsuperscript{th} place
defines a decreasing filtration 
on \eqref{equation:resolution-holomorphic} 
on the level of complexes and
they induce a decreasing filtration \(\mathrm{F}^{\bullet}\)
on the cohomology group \(\mathrm{H}^{k}(X;\mathbb{C})\);
specifically \(\mathrm{F}^{p}\mathrm{H}^{k}(X;\mathbb{C})\)
is defined to be the image of the canonical map
\begin{equation*}
\mathbf{H}^{k}(X;\mathrm{F}^{p}\Omega_{X}^{\bullet})\to
\mathbf{H}^{k}(X;\Omega_{X}^{\bullet})=\mathrm{H}^{k}(X;\mathbb{C})
\end{equation*}
induced from \(\mathrm{F}^{p}\Omega_{X}^{\bullet}\to \Omega_{X}^{\bullet}\).
\begin{definition}
We say that the
\emph{\(\partial\bar{\partial}\)-lemma holds on \(X\)}
or call \(X\) \emph{a \(\partial\bar{\partial}\)-manifold}
if for all \(p,q\)
and for all \(\mathrm{d}\)-closed \((p,q)\)-form \(\eta\in A^{p,q}(X)\),
the following statement holds:
\begin{equation} 
\tag{\(\ast\)}
\mbox{\(\eta\) is \(\mathrm{d}\)-exact, 
\(\partial\)-exact, or \(\bar{\partial}\)-exact 
\(\Leftrightarrow\)
\(\eta = \partial\bar{\partial}\xi\) for some \(\xi\in A^{p-1,q-1}(X)\)}.
\end{equation}
\end{definition}

\begin{remark}
It is known that the \(\partial\bar{\partial}\)-lemma
is equivalent to the following statement.
Let \(\eta\in A^{k}(X)\) be a both \(\partial\)-closed
and \(\bar{\partial}\)-closed \(k\)-form. Then 
\begin{equation*} 
\mbox{\(\eta\) is \(\mathrm{d}\)-exact
\(\Leftrightarrow\)
\(\eta = \partial\bar{\partial}\xi\) for some \(\xi\in A^{k-2}(X)\)}.
\end{equation*}
\end{remark}

It is well-known (\cite{1977-Deligne-Griffiths-Morgan-Sullivan-real-homotopy-theory-of-kahler-manifolds}*{Equation (5.21)}
and \cite{1971-Deligne-theorie-de-Hodge-II}*{Equation (4.3.1)})
that the following two statements 
are equivalent.
\begin{itemize}
\item[\textbf{(i)}] 
\begin{enumerate}
\item[(a)] The Hodge-to-de~Rham spectral sequence 
(a.k.a.~the Fr\"{o}licher spectral sequence)
on \(X\) degenerates at \(\mathrm{E}_{1}\);
\item[(b)] for all \(k\), if \(\mathrm{F}^{\bullet}\) denotes 
the corresponding filtration on \(\mathrm{H}^{k}(X;\mathbb{C})\),
then \(\mathrm{F}^{\bullet}\) and \(\bar{\mathrm{F}}^{\bullet}\) 
are \(k\)-opposed; namely
\(\mathrm{F}^{p}\oplus\bar{\mathrm{F}}^{k+1-p}=\mathrm{H}^{k}(X;\mathbb{C})\).
\end{enumerate} 

\item[\textbf{(ii)}] The \(\partial\bar{\partial}\)-lemma holds on \(X\).  
\end{itemize}

For a compact complex manifold, consider the short exact sequence
\begin{eqnarray}
0\to \mathrm{F}^{p+1}\Omega_{X}^{\bullet}\to
\mathrm{F}^{p}\Omega_{X}^{\bullet}\to \mathrm{F}^{p}\Omega_{X}^{\bullet}
\slash\mathrm{F}^{p+1}\Omega_{X}^{\bullet}\cong\Omega_{X}^{p}[-p]\to 0.
\end{eqnarray}
Taking hypercohomology, we obtain a long exact sequence
\begin{eqnarray*}
\cdots\to \mathbf{H}^{k}(X;\mathrm{F}^{p+1}\Omega_{X}^{\bullet})\to 
\mathbf{H}^{k}(X;\mathrm{F}^{p}\Omega_{X}^{\bullet})\to
\mathrm{H}^{k-p}(X;\Omega_{X}^{p})\to
\mathbf{H}^{k+1}(X;\mathrm{F}^{p+1}\Omega_{X}^{\bullet})\to\cdots.
\end{eqnarray*}
If \(X\) is a \(\partial\bar{\partial}\)-manifold, there exists a canonical injection
\begin{eqnarray*}
\mathrm{H}^{p,q}(X):=\mathrm{H}^{q}(X;\Omega_{X}^{p})\to \mathrm{H}^{p+q}(X;\mathbb{C})
\end{eqnarray*}
from the Dolbeault cohomology to the de Rham cohomology.
For a Dolbeault class \([\theta]\in\mathrm{H}^{p,q}(X)\), we choose a \(\bar{\partial}\)-closed
representative \(\theta\in A^{p,q}(X)\).
Now \(\partial\theta\in A^{p+1,q}(X)\) is \(\mathrm{d}\)-closed and clearly \(\partial\)-exact.
By assumption, there exists \(v\in A^{p,q-1}(X)\) such that 
\(\partial\theta=\partial\bar{\partial}v\). So \(\theta-\bar{\partial}v\)
is a \(\mathrm{d}\)-closed representative of \([\theta]\).

To prove the injectivity, let \([\theta]\in\mathrm{H}^{p,q}(X)\)
such that \(\theta-\bar{\partial}v\) is also \(\mathrm{d}\)-exact.
By assumption again, there exists \(u\in A^{p-1,q-1}(X)\) such that
\begin{eqnarray*}
\theta-\bar{\partial}v = \partial\bar{\partial}u.
\end{eqnarray*}
Thus \(\theta = \bar{\partial}(v+\partial u)\) is \(\bar{\partial}\)-exact
and hence \([\theta]=0\in \mathrm{H}^{p,q}(X)\).
It is also clear that 
\begin{eqnarray*}
\mathrm{H}^{p,q}(X) = \overline{\mathrm{H}^{q,p}(X)}
\end{eqnarray*}
when we regard them as subspaces in \(\mathrm{H}^{p+q}(X;\mathbb{C})\). In particular, 
\(\mathrm{H}^{k}(X;\mathbb{C})\) carries a weight \(k\) Hodge structure.



\subsection{Calabi--Yau \(\partial\bar{\partial}\)-manifolds}
\begin{definition}
Let \(X\) be a compact complex manifold of dimension \(n\). We
say that \(X\) is a \emph{Calabi--Yau (CY) \(\partial\bar{\partial}\)-manifold}
if \(X\) is a \(\partial\bar{\partial}\)-manifold whose 
canonical bundle \(\Omega^{n}_{X}\) is trivial.
\end{definition}
Many theorems stated for CY manifolds (compact K\"{a}hler manifolds with
trivial canonical bundle) also hold for CY \(\partial\bar{\partial}\)-manifold.
For instance, we have the unobstructedness theorem, 
the Bogomolov--Tian--Todorov theorem for CY \(\partial\bar{\partial}\)-manifolds.
\begin{theorem}[\cite{2019-Popovici-holomorphic-deformations-of-balanced-calabi-yau-d-dbar-manifolds}*{Theorem 1.2}]
Let \(X\) be a compact CY \(\partial\bar{\partial}\)-manifold.
Then \(X\) has unobstructed deformation, i.e.~the Kuranishi space of \(X\) is smooth.
\end{theorem}
We also have the following local Torelli theorem.
\begin{theorem}[\cite{2019-Popovici-holomorphic-deformations-of-balanced-calabi-yau-d-dbar-manifolds}*{Theorem 5.4}]
Let \(X\) be a compact CY \(\partial\bar{\partial}\)-manifold of
complex dimension \(n\) and let \(\pi\colon\mathcal{X}\to S\) be its 
Kuranishi family. Then the associated period map
\begin{eqnarray*}
\mathcal{P}\colon S\to D\subset\mathbf{P}\mathrm{H}^{n}(X;\mathbb{C}),~
s\mapsto \mathrm{H}^{n,0}(\mathcal{X}_{s};\mathbb{C})
\end{eqnarray*}
is a local holomorphic immersion.
\end{theorem}

Let us recall a definition for the later use.
\begin{definition}
Let \((X,\omega)\) be a balanced CY \(\partial\bar{\partial}\)-manifold.
Given a Dolbeault cohomology class \([\theta]\in\mathrm{H}^{p,q}(X)\), 
let \(\theta\in A^{p,q}(X)\)be a \(\Delta_{\omega}''\)-representative
and \(v_{\mathrm{min}}\in A^{p,q-1}(X)\) such that \(v_{\mathrm{min}}\)
is the solution to
\(\partial\theta=\partial\bar{\partial}v\) with minimum \(L^{2}\) norm
with respect to \(\omega\). The \(\mathrm{d}\)-closed form
\(\theta_{\mathrm{min}}:=\theta+\bar{\partial}v_{\mathrm{min}}\)
is called the \emph{\(\omega\)-minimal \(\mathrm{d}\)-closed
representative of \([\theta]\)}.
\end{definition}


\subsection{Deligne's canonical extension}
\label{subsec:deligne-ext}
Let \(\mathcal{V}\to \Delta^{\ast}\) be a holomorphic vector bundle with a flat connection \(\nabla\) whose monodromy \(T\) is unipotent. 
Denote by 
\begin{eqnarray*}
N:=\log T=\log (I-(I-T))=-\sum_{k=1}^{\infty}\frac{(I-T)^{k}}{k}.
\end{eqnarray*} 
the associated nilpotent operator so that \(T=e^{N}\).
Denote by 
\begin{eqnarray*}
\pi\colon \mathfrak{h}\to\Delta^{\ast},~z\mapsto t=e^{2\pi\sqrt{-1}z}, 
\end{eqnarray*}
the universal cover.
Then the pullback \(\pi^{\ast}\mathcal{V}\) becomes a trivial vector bundle.
Let \(\{v_{1},\ldots,v_{m}\}\) be a flat trivialization for \(\pi^{\ast}\mathcal{V}\).
For a flat section \(v\) of \(\mathcal{\pi^{\ast}V}\), we see that it satisfies the relation
\(v(z+1) = Tv(z)\). Therefore
the twisted section
\begin{equation*}
u:=e^{-zN} v
\end{equation*}
is invariant under the translation \(z\mapsto z+1\), for
\begin{eqnarray*}
u(z+1) = e^{-(z+1)N} v(z+1) = e^{-zN} e^{-N} Tv(z) = e^{-zN}v(z) = u(z).
\end{eqnarray*}
It follows that \(u\) becomes a (non-flat) section of \(\mathcal{V}\) on \(\Delta^{\ast}\).
In this manner, we obtain a twisted section \(\{u_{1},\ldots,u_{m}\}\)
for \(j_{\ast}\mathcal{V}\) where \(j\colon\Delta^{\ast}\to\Delta\)
is the open inclusion.
We can thus extend \(\mathcal{V}\)
to a holomorphic vector bundle 
\(\bar{\mathcal{V}}\) on \(\Delta\) via the frame \(\{u_{1},\ldots,u_{m}\}\). 
We compute 
\begin{eqnarray*}
\nabla_{t\partial_{t}} u_{k} = \frac{1}{2\pi\sqrt{-1}}
\nabla_{\partial_{z}} u_{k} = \frac{-N u_{k}}{2\pi\sqrt{-1}}.
\end{eqnarray*}
We obtain a logarithmic connection \(\bar{\nabla}\) on \(\bar{\mathcal{V}}\)
with residue \(-N\slash 2\pi\sqrt{-1}\) along \(t=0\), i.e.
\begin{eqnarray*}
\nabla = \mathrm{d} - \frac{N}{2\pi\sqrt{-1}}\otimes\frac{\mathrm{d}t}{t}.
\end{eqnarray*}
In which case, the connection \(\nabla\)
is extended to be a connection \(\bar{\nabla}\)
on \(\bar{\mathcal{V}}\)
with a logarithmic pole along \(t=0\).
The extension \((\bar{\mathcal{V}},\bar{\nabla})\)
is called \emph{Deligne's canonical extension}.

\subsection{Mixed Hodge structures and Deligne's splitting}
Let \(H_{\mathbb{R}}\) be a finite-dimensional real vector space. 
Put \(H:=H_{\mathbb{R}}\otimes_{\mathbb{R}}\mathbb{C}\).
Recall that a (real) mixed Hodge structure (\(\mathbb{R}\)-mixed Hodge structure)
on \(H_{\mathbb{R}}\)
consists of a pair of finite filtrations
\begin{align*}
\cdots\subset W_{k-1}\subset W_{k}\subset W_{k+1}\subset\cdots\\
\cdots\supset F^{p-1}\supset F^{p}\supset F^{p+1}\supset\cdots
\end{align*}
such that
\begin{itemize}
\item[(i)] \(W\) is defined over \(\mathbb{R}\);
\item[(ii)] for \(l\in\mathbb{Z}\), the filtration \(F^{\bullet}\) induces
a Hodge structure of weight \(l\)
on the graded piece \(\mathrm{Gr}_{l}^{W}H\).
\end{itemize}
The increasing filtration \(W_{\bullet}\) is called the \emph{weight filtration}
whereas the decreasing filtration \(F^{\bullet}\) is referred as the \emph{Hodge filtration}.

Recall that for a given \(\mathbb{R}\)-mixed Hodge structure 
\((H,W_{\bullet},F^{\bullet})\),
\emph{a weak spliiting of \((H,W_{\bullet},F^{\bullet})\)} is a decomposition 
of the vector space
\begin{equation}
H = \bigoplus_{p,q} J^{p,q}
\end{equation}
such that
\begin{equation}
\label{eq:mixed-hodge-splitting}
W^{\mathbb{C}}_{k} = \bigoplus_{p+q\le k} J^{p,q}~\mbox{and}~F^{p}=\bigoplus_{r\ge p} J^{r,q}_{\mathbb{C}}.
\end{equation}
Notice that in \eqref{eq:mixed-hodge-splitting} the
notation \(W^{\mathbb{C}}_{k}\) stands for
\(W_{k}\otimes_{\mathbb{R}}\mathbb{C}\).

By a theorem of Deligne, any \(\mathbb{R}\)-mixed Hodge structure always admits 
such a splitting. This is called the \emph{Deligne's splitting}. In fact,
given a \(\mathbb{R}\)-mixed Hodge stuctrue \((H,W_{\bullet},F^{\bullet})\), Delinge's splitting
is defined by
\begin{equation}
I^{p,q}:= F^{p}\cap W^{\mathbb{C}}_{p+q}\cap \left(\overline{F^{q}}\cap W^{\mathbb{C}}_{p+q}+
\sum_{j\ge 2}\overline{F^{q-j+1}}\cap W^{\mathbb{C}}_{p+q-j}\right).
\end{equation}
One checks that they satisfy \eqref{eq:mixed-hodge-splitting}
and also the property 
\begin{equation}
\label{eq:deligne-splitting}
I^{p,q} = \overline{I^{q,p}} \mod{\bigoplus_{r<p,~s<q} I^{r,s}}.
\end{equation}
For a proof, one can consult \cite{1986-Cattani-Kaplan-Schmid-degeneration-of-hodge-structures}*{(2.13) Theorem}.

\subsection{Steenbrink's limiting mixed Hodge structure}
\label{subsec:st-limiting-mhs}
We begin with the following setup.
\begin{itemize}
    \item Let \(f\colon \mathcal{X}\to \Delta\) be a semi-stable model, 
    i.e.~\(\mathcal{X}\) is a complex manifold, \(f\) is smooth over the punctured disk \(\Delta^{\ast}\),
    and the central fiber \(f^{-1}(0)\) is a simple normal crossing divisor.
    \item Let \(E_{1},\ldots,E_{m}\) be the irreducible components of \(E:=f^{-1}(0)\).
    \item For any subset \(I\subset \{1,\ldots,m\}\), we put
    \begin{equation*}
        E_{I}:=\bigcap_{i\in I} E_{i}~\mbox{and}~E(k):=\coprod_{|I|=k} E_{I}.
    \end{equation*}
    We have proper maps \(\iota_{k}\colon E(k)\to E\) for \(k\ge 1\).
    We shall also set \(E(0) = E\).
    \item For each \(I=\{i_{1}<\ldots<i_{r}\}\) and \(1\le j\le r\), 
    put \(I_{j}:=I\setminus \{i_{j}\}\); the subset of \(I\) with
    \(j\)\textsuperscript{th}
    element being omitted.
    Denote by \(\iota_{I,I_{j}}\colon E_{I}\to E_{I_{j}}\) the inclusion map.
    We have the restriction map \(\iota_{I,I_{j}}^{\ast}\colon \mathrm{H}^{k}(E_{I_{j}};\mathbb{C})
    \to\mathrm{H}^{k}(E_{I};\mathbb{C})\) and the
    Gysin map \((\iota_{I,I_{j}})_{!}\colon \mathrm{H}^{k}(E_{I};\mathbb{Z})
    \to\mathrm{H}^{k+2}(E_{I_{j}};\mathbb{C})\).
    Let \(\iota_{r,j}:=\oplus_{|I|=r} \iota_{I,I_{j}}\) and put
    \begin{eqnarray*}
    \psi_{r}:=\bigoplus_{j=1}^{r} (-1)^{j-1}\iota_{r,j}^{\ast}~\mbox{and}~
    \phi_{r}:=-\bigoplus_{j=1}^{r} (-1)^{j-1}(\iota_{r,j})_{!}.
    \end{eqnarray*}
    They are refereed as an \emph{alternating restriction map} and
    an \emph{alternating Gysin map}.
    We shall skip the subscript \(r\) when no confussion occurs.
    \item Let \(i_{t}\colon \mathcal{X}_{t}\to\mathcal{X}\) 
    be the inclusion of a fiber \(\mathcal{X}_{t}\).
    The total space \(\mathcal{X}\) is homotopic to \(E\) by a deformation retraction
    \(r\colon \mathcal{X}\to E\).
    Denote by \(r_{t}=r\circ i_{t}\colon\mathcal{X}_{t}\to E\).
    The complex \((Rr_{t})_{\ast}i_{t}^{\ast}\mathbb{Z}_{\mathcal{X}}\) is called 
    \emph{the complex of nearby cocycles}.
\end{itemize}
The hypercohomology of the complex \((Rr_{t})_{\ast}i_{t}^{\ast}\mathbb{Z}_{\mathcal{X}}\)
computes the cohomology of the nearby fiber. In fact, we have
\begin{equation*}
    \mathbf{H}^{q}((Rr_{t})_{\ast}i_{t}^{\ast}\mathbb{Z}_{\mathcal{X}})
    = \mathbf{H}^{q}(i_{t}^{\ast}\mathbb{Z}_{\mathcal{X}}) 
    = \mathrm{H}^{q}(i_{t}^{\ast}\mathbb{Z}_{\mathcal{X}})
    = \mathrm{H}^{q}(\mathcal{X}_{t}).
\end{equation*}
Nevertheless, we shall recall a concrete recipe to calculate 
the cohomology of the nearby fiber.

Let \(\pi\colon\mathfrak{h}\to \Delta^{\ast}\) be the universal cover.
The fiber product \(\mathcal{X}_{\infty}:=\mathcal{X}\times_{\Delta^{\ast}}\mathfrak{h}\)
is called the \emph{canonical fiber}. Note that \(\mathcal{X}_{\infty}\)
is homotopic to any smooth fiber of \(f\).

Look at the commutative diagram
\begin{equation*}
    \begin{tikzcd}
        & \mathcal{X}_{\infty} \ar[r,"p"]\ar[d,"F"]\ar[rr,bend left=40,"k"] & \mathcal{X}^{\ast}\ar[r,"j"]\ar[d] &\mathcal{X}\ar[d,"f"] & E\ar[l,"i"]\ar[d]\\
        & \mathfrak{h}\ar[r,"\pi"]\ar[rr,bend right=40,"\kappa"'] &\Delta^{\ast}\ar[r] &\Delta & \{0\}\ar[l].
    \end{tikzcd}
\end{equation*}
We can form the complex \(i^{\ast}(Rk)_{\ast}k^{\ast}\mathbb{Z}_{\mathcal{X}}\)
which computes the cohomology of the nearby fiber.

\begin{definition}
    Let \(\mathscr{K}^{\bullet}\in D^{+}(\mathcal{X},A)\) be a bounded below complex
    of sheaves of \(A\)-modules on \(\mathcal{X}\). We define
    the \emph{nearby cocycle}, denoted by \(\Phi_{f}\mathscr{K}^{\bullet}\),
    to be the complex
    \begin{equation*}
        \Phi_{f}\mathscr{K}^{\bullet} := i^{\ast} (Rk)_{\ast} k^{\ast} \mathscr{K}^{\bullet}
        \in D^{+}(E,A).
    \end{equation*}
There is a canonical morphism
\begin{equation*}
    i^{\ast}\mathscr{K}^{\bullet}\to \Phi_{f}\mathscr{K}^{\bullet}.
\end{equation*}
Let \(\phi_{f}\mathscr{K}^{\bullet}\) be the mapping cone the 
the morphism above. We obtain a triangle
\begin{equation*}
    i^{\ast}\mathscr{K}^{\bullet}\to \Phi_{f}\mathscr{K}^{\bullet}\to \phi_{f}\mathscr{K}^{\bullet}
    \xrightarrow[]{+1}
\end{equation*}
and \(\phi_{f}\mathscr{K}^{\bullet}\) is 
the complex of \emph{vanishing cocycles}.
\end{definition}
To define a mixed Hodge structure on 
\(\Phi_{f}\mathbb{C}_{\mathcal{X}}\), we
consider the \emph{relative logarithmic de Rham complex}
\(\Omega^{\bullet}_{\mathcal{X}\slash \Delta}(\log E)\).
Explicitly, it is defined as the cokernel of the morphism
\begin{eqnarray*}
\begin{tikzcd}[column sep=4em]
&\Omega^{\bullet-1}_{\mathcal{X}}(\log E)\ar[r,"\wedge f^{\ast}(\mathrm{d}t\slash t)"]
&\Omega^{\bullet}_{\mathcal{X}}(\log E).
\end{tikzcd}
\end{eqnarray*}
Let \(\mathcal{I}\subset\mathscr{O}_{\mathcal{X}}\) 
be the ideal sheaf defining \(E\).
One can check that \(\mathcal{I}\Omega^{\bullet}_{\mathcal{X}\slash \Delta}(\log E)\)
is stable under the differential.
We thus obtain a complex 
\(\Omega^{\bullet}_{\mathcal{X}\slash \Delta}(\log E)\otimes\mathscr{O}_{E}\).
\begin{theorem}
There exists a quasi-isomorphism in \(D^{+}(\mathbb{C}_{E})\)
\begin{equation*}
    \Phi_{f}\mathbb{C}_{\mathcal{X}}\simeq
    \Omega^{\bullet}_{\mathcal{X}\slash \Delta}(\log E)\otimes\mathscr{O}_{E}.
\end{equation*}
\end{theorem}
Our goal is to define a mixed Hodge structure on \(\Phi_{f}\mathbb{C}_{\mathcal{X}}\)
via \(\Omega^{\bullet}_{\mathcal{X}\slash \Delta}(\log E)\otimes\mathscr{O}_{E}\)
at the level of complexes. We need two pieces of information:
(a) the Hodge filtration \(\mathcal{F}^{\bullet}\) 
and (b) the weight filtration \(\mathcal{W}_{\bullet}\).
For this purpose, we shall henceforth assume that 
\(E_{I}\) supports the \(\partial\bar{\partial}\)-lemma
for each subset \(I\).

\subsubsection{The Hodge filtration}
This part is easy; we may use the stupid filtration \(\mathcal{F}^{\bullet}\) on the
complex \(\Omega^{\bullet}_{\mathcal{X}\slash \Delta}(\log E)\otimes\mathscr{O}_{E}\),
i.e.
\begin{equation*}
    \mathcal{F}^{k}:=\Omega^{\bullet\ge k}_{\mathcal{X}\slash \Delta}(\log E)\otimes\mathscr{O}_{E}.
\end{equation*}
By the \(\partial\bar{\partial}\)-lemma, this puts a Hodge filtration 
on the hypercohomology.

\subsubsection{The monodromy weight filtration}

There is a natual 
filtration \(\mathcal{W}_{\bullet}\) on \(\Omega^{\bullet}_{\mathcal{X}}(\log E)\)
\begin{equation*}
    \mathcal{W}_{p}(\Omega^{\bullet}_{\mathcal{X}\slash \Delta}(\log E)\otimes\mathscr{O}_{E})
    :=\mbox{The image of}~\mathcal{W}_{p}\Omega^{\bullet}_{\mathcal{X}}(\log E)
\end{equation*}
under the canonical projection \(\Omega^{\bullet}_{\mathcal{X}}(\log E)\to 
\Omega^{\bullet}_{\mathcal{X}\slash \Delta}(\log E)\otimes\mathscr{O}_{E}\).

Consider the double complex \((\mathcal{D}^{\bullet,\bullet},D_{1},D_{2})\) with
\begin{align*}
    \mathcal{D}^{p,q}:&=\Omega^{p+q+1}_{\mathcal{X}}(\log E)\otimes\mathscr{O}_{E}
    \slash \mathcal{W}_{p}(\Omega^{p+q+1}_{\mathcal{X}}(\log E)\otimes\mathscr{O}_{E})\\
    &\left(\cong 
    \Omega^{p+q+1}_{\mathcal{X}}(\log E)\slash \mathcal{W}_{p}(\Omega^{p+q+1}_{\mathcal{X}}(\log E))~
    \mbox{whenever}~p\ge 0\right)
\end{align*}
for \(p,q\ge 0\)
with differentials
\begin{equation*}
    \begin{cases}
        D_{1}\colon \mathcal{D}^{p,q}\to\mathcal{D}^{p+1,q}~~&\mbox{via}~\omega\mapsto (\mathrm{d}t\slash t)\wedge\omega\\
        D_{2}\colon \mathcal{D}^{p,q}\to\mathcal{D}^{p,q+1}~~&\mbox{via}~\omega\mapsto \mathrm{d}\omega.
    \end{cases}
\end{equation*}
Define \(\mu\colon \Omega^{q}_{\mathcal{X}\slash\Delta}(\log E)\otimes\mathscr{O}_{E}
\to \mathcal{D}^{0,q}\) via
\begin{equation*}
    \omega\mapsto (-1)^{q}\frac{\mathrm{d}t}{t}\wedge\omega~\mod{\mathcal{W}_{0}\Omega_{\mathcal{X}}^{q+1}(\log E)}.
\end{equation*}
With this definition, we see that
\(\mu\) induces a morphism of complexes
\begin{equation*}
    \Omega^{\bullet}_{\mathcal{X}\slash\Delta}(\log E)\otimes\mathscr{O}_{E} 
    \to \mathrm{Tot}(\mathcal{D}^{\bullet,\bullet}).
\end{equation*}
Let us define an increasing filtration 
\(\mathcal{W}_{\bullet}\) on \(\mathcal{D}^{\bullet,\bullet}\) by
\begin{equation*}
    \mathcal{W}_{r}\mathcal{D}^{p,q}:=\mbox{the image of \(\mathcal{W}_{r+p+1}
    (\Omega^{p+q+1}_{\mathcal{X}}(\log E)\otimes\mathscr{O}_{E})\) in \(\mathcal{D}^{p,q}\)}.
\end{equation*}
Also, let us define the Hodge filtration on \(\mathcal{D}^{p,q}\) via
\begin{equation*}
    \mathcal{F}^{k}(\mathrm{Tot}(\mathcal{D}^{\bullet,\bullet})):=\bigoplus_{p}\bigoplus_{q\ge k}
    \mathcal{D}^{p,q}.
\end{equation*}

\begin{theorem}
\label{thm:quasi-iso}
\(\mu\) is a quasi-isomorphism between bi-filtered complexes
\begin{equation*}
(\Omega^{\bullet}_{\mathcal{X}\slash\Delta}(\log E)\otimes
\mathscr{O}_{E},\mathcal{W}_{\bullet},\mathcal{F}^{\bullet})~\mbox{and}~    
(\mathrm{Tot}(\mathcal{D}^{\bullet,\bullet}),\mathcal{W}_{\bullet},\mathcal{F}^{\bullet}).
\end{equation*}
\end{theorem}

\begin{definition}[The monodromy weight filtration]
    We define another increasing filtration \(\mathcal{W}(M)_{\bullet}\)
    on the complex \(\mathcal{D}^{p,q}\) via
    \begin{equation*}
        \mathcal{W}(M)_{r}\mathcal{D}^{p,q}:=
        \mbox{the image of \(\mathcal{W}_{r+2p+1}
    (\Omega^{p+q+1}_{\mathcal{X}}(\log E)\otimes\mathscr{O}_{E})\) in \(\mathcal{D}^{p,q}\)}.
    \end{equation*}
The filtration \(\mathcal{W}(M)_{\bullet}\) is called 
the \emph{monodromy weight filtration}.
\end{definition}

\begin{remark}
    As we shall see, the punchline is that \(D_{1}\) becomes the zero map
    in the associated graded module with respect to \(\mathcal{W}(M)_{\bullet}\).
    
\end{remark}
    
    Note that \(\mathcal{W}(M)_{r}\mathcal{D}^{p,q}=0\) whenever \(r+2p+1\le p\),
    i.e.~\(r\le -1-p\).
    Similarly, we have
    \begin{equation*}
        \mathrm{Gr}_{r}^{\mathcal{W}(M)}(\mathcal{D}^{p,q})=
        \mathrm{Gr}_{r+2p+1}^{\mathcal{W}}\Omega^{p+q+1}_{\mathcal{X}}(\log E)
        \cong (\iota_{r+2p+1})_{\ast}\Omega_{E(r+2p+1)}^{q-p-r}
    \end{equation*}
    and the last isomorphism holds when \(r+2p\ge 0\).
    In particular, this implies 
    \begin{equation*}
        \mathrm{Gr}_{r}^{\mathcal{W}(M)}(\mathcal{D}^{p,q}) = 0
    \end{equation*}
    whenever \(r\ge q-p+1\ge -2p\) or \(r\le -p-1\).
    Because of this shift, the morphism \(D_{1}\)
    becomes zero on the associated graded complex.
    The \((p,q)\) term in the associated graded complex 
    \(\mathrm{Gr}_{r}^{\mathcal{W}(M)}(\mathrm{Tot}(\mathcal{D}^{\bullet,\bullet}))\) becomes
    \begin{equation*}
        \bigoplus_{-p\le r\le q-p} \mathrm{Gr}_{r+2p+1}^{\mathcal{W}}\Omega^{p+q+1}_{\mathcal{X}}(\log E)
        \cong \bigoplus_{-p\le r\le q-p} (\iota_{r+2p+1})_{\ast}\Omega^{p+q}_{E(r+2p+1)}[-r-2p]
    \end{equation*}
    and the 
    cohomology of \((\mathrm{Gr}_{r}^{\mathcal{W}(M)}(\mathrm{Tot}(\mathcal{D}^{\bullet,\bullet})),
    D_{1}+D_{2}=D_{2})\)
    is \(\mathbb{C}_{E(r+2p+1)}[-r-2p]\),
    which computes the cohomology
    of the smooth variety \(E(r+2p+1)\) up to a shift.
    We remind the reader that the summation is taken over all
    non-negative integers \(p\) and \(q\) satisfying \(-p\le r\le q-p\).
    It is now obvious that we can define the Hodge structure using
    the usual Hodge filtration on smooth varieties.

\begin{instance}[Conifold transitions]
    Let \(X\nearrow Y\) be a conifold transition
    between Calabi--Yau threefolds, i.e.~\(X\) and \(Y\) are smooth
    Calabi--Yau threefolds and there exist
    a complex degeneration \(X\rightsquigarrow \bar{X}\)
    to a conifold threefold and a small resolution \(Y\to \bar{X}\). 
    Let \(\{p_{1},\ldots,p_{k}\}\) be the set of singular points in \(\bar{X}\)
    and \(C_{1},\ldots,C_{k}\) be the rational curves in \(Y\) 
    lying over \(p_{1},\ldots,p_{k}\).

    By Friedman's result, the smoothing corresponds to a non-trivial relation
    \begin{equation*}
        \sum_{i=1}^{k} m_{i}[C_{i}] = 0 \in\mathrm{H}_{2}(Y,\mathbb{C})~\mbox{with}~m_{i}\ne 0~\mbox{for all}~i.
    \end{equation*}
    Let \(\pi\colon\mathcal{X}'\to\Delta\) be the 
    complex degeneration. To obtain a semi-stable model, we first perform a degree two base change and 
    then blow up all the singularities on the central fiber. 
    Denote by \(f\colon \mathcal{X}\to \Delta\)
    the resulting semi-stable family.
    \begin{equation*}
    \begin{tikzcd}
        &\mathcal{X}\ar[r]\ar[d,"f"']&\mathcal{X}'\ar[d,"\pi"]\\
        &\Delta\ar[r] &\Delta.
    \end{tikzcd}    
    \end{equation*}
    The central fiber \(f^{-1}(0)=\tilde{Y}\bigcup \cup_{i=1}^{r} Q_{i}\)
    is a simple normal crossing divisor. 
    Here
    \begin{equation*}
        \tilde{Y} = \mathrm{Bl}_{\sqcup_{i=1}^{r} C_{i}} Y~\mbox{and}~Q_{i}\subset \mathbf{P}^{4}~\mbox{is
        a smooth quadric hypersurface}.
    \end{equation*}
    Moreover, the exceptional divisor \(E_{i}=\tilde{Y}\cap Q_{i}\) is 
    isomorphic to a quadric in \(\mathbf{P}^{3}\).

    We can compute the graded complex \(\mathrm{Gr}_{r}^{\mathcal{W}(M)}\mathcal{D}^{p,q}\) in this case. We have
    \begin{equation*}
        E(1) = \tilde{Y}\coprod \sqcup_{i=1}^{r} Q_{i},~E(2) = \coprod_{i=1}^{r} E_{i},~\mbox{and}~E(k)=\emptyset~\mbox{for}~k\ge 3.
    \end{equation*}
    We wish to compute the group \(\mathrm{Gr}^{\mathcal{W}(M)}_{4}\mathrm{H}^{3}(X,\mathbb{C})\)
    which is the cohomology of the complex
    \begin{equation*}
        E_{1}^{-2,4}=\mathbf{H}^{2}(\mathrm{Gr}_{2}^{\mathcal{W}(M)})\to E_{1}^{-1,4}
    =\mathbf{H}^{3}(\mathrm{Gr}_{1}^{\mathcal{W}(M)})\to E_{1}^{0,4}=\mathbf{H}^{4}(\mathrm{Gr}_{0}^{\mathcal{W}(M)}).
    \end{equation*}
    A similar calculation shows that 
    \begin{equation*}
        \mathrm{Gr}_{1}^{\mathcal{W}(M)}\cong \mathbb{C}_{E(2)}[-1],~\mbox{and}~
        \mathrm{Gr}_{0}^{\mathcal{W}(M)}\cong \mathbb{C}_{E(1)}
    \end{equation*}
    and the sequence is transformed into
    \begin{equation*}
        0\to \mathrm{H}^{2}(E(2),\mathbb{C})=\bigoplus_{i=1}^{r}\mathrm{H}^{2}(E_{i},\mathbb{C})
        \to \mathrm{H}^{4}(E(1),\mathbb{C})=\mathrm{H}^{4}(\tilde{Y},\mathbb{C})\oplus 
        \bigoplus_{i=1}^{r}\mathrm{H}^{4}(Q_{i},\mathbb{C}).
    \end{equation*}
    This is dual to the sequence in \cite{2018-Lee-Lin-Wang-towards-A+B-theory-in-conifold-transitions-for-CY-threefolds}*{Equation (1.4)}.
    
\end{instance}

\begin{theorem}
\label{theorem:steenbrink-mhs}
    The triple \((\mathrm{Tot}(\mathcal{D}^{\bullet,\bullet}),\mathcal{W}(M)_{\bullet},\mathcal{F}^{\bullet})\)
    forms a mixed Hodge structure. 
    Consequently, by Theorem \ref{thm:quasi-iso},
    this puts a mixed Hodge structure on 
    \(\Omega^{\bullet}_{\mathcal{X}\slash\Delta}(\log E)\otimes
    \mathscr{O}_{E}\sim \Phi_{f}\mathbb{C}_{\mathcal{X}}\).
\end{theorem}

\begin{corollary}
    The cohomology of the derived pushforward
    \begin{equation*}
        \mathbf{R}f_{\ast} \Omega^{\bullet}_{\mathcal{X}\slash\Delta}(\log E)
    \end{equation*}
    is a locally free sheaf.
\end{corollary}
\begin{corollary}
    The sheaf
    \begin{equation*}
        R^{q}f_{\ast} \Omega^{p}_{\mathcal{X}\slash\Delta}(\log E)
    \end{equation*}
    is locally free. Consequently, 
    \(\mathbf{R}f_{\ast}\mathcal{F}^{k}\) gives a locally free extension of the Hodge bundles.
\end{corollary}

\section{The period-map metric
on the moduli of CY \texorpdfstring{\(\partial\bar{\partial}\)}{ddbar}-manifolds}

\subsection{Weil--Petersson metric and the period-map metric}
For a compact K\"{a}hler CY manifold \((X,\omega)\) 
of complex dimension \(n\), it is known that
its Kuranishi space \(S\) is smooth. 
Let \(\mathcal{X}\to S\)
be a local universal deformation of \(X\) 
with a fixed K\"{a}hler class \([\omega]\in \mathrm{H}^{2}(X;\mathbb{Z})\).
There exists a canonical metric on \(S\), called
the Weil--Petersson metric. Indeed,
Yau's theorem 
\cite{1978-Yau-on-the-ricci-curvature-of-a-compact-kahler-manifold-and-the-complex-monge-ampere-equation-i}
provides us with a Ricci flat metric \(g_{s}\)
on \(\mathcal{X}_{s}\) whose K\"{a}hler class 
equals \([\omega]\).
Let
\(\rho\colon T_{s}S\to\mathrm{H}^{1}(\mathcal{X}_{s},T\mathcal{X}_{s})\) 
be the Kodaira--Spencer map at \(s\in S\). 
Identify \(\mathrm{H}^{1}(\mathcal{X}_{s},T\mathcal{X}_{s})\) 
with \(\mathbb{H}^{0,1}(T\mathcal{X}_{s})\)
(the \(T\mathcal{X}_{s}\)-valued harmonic \((0,1)\)-forms w.r.t.~\(g_{s}\)).
For \(\lambda,\theta\in T_{s}S\),
one can check that the pairing
\begin{eqnarray*}
G_{s}(\lambda,\theta):=\int_{\mathcal{X}_{s}} \langle\rho(\lambda),\rho(\theta)\rangle_{s}
\end{eqnarray*}
is positive definite and hence defines a metric on \(S\).
Here \(\langle-,-\rangle_{s}\) is the product
on \(\mathbb{H}^{0,1}(T\mathcal{X}_{s})\) induced by \(g_{s}\). 
The metric \(G\) is called the \emph{Weil--Petersson metric};
this is a K\"{a}hler metric on \(S\).

Denote by \(Q\colon\mathrm{H}^{n}(X;\mathbb{C})\times
\mathrm{H}^{n}(X;\mathbb{C})\to\mathbb{C}\) the topological pairing
\begin{eqnarray}
\label{eq:topological-pairing}
Q([\alpha],[\beta]) = (-1)^{n(n-1)/2}\int_{X}\alpha\wedge\beta.
\end{eqnarray}
From Hodge theory, the sesquilinear form
\begin{eqnarray}
\label{eq:sesuilinear-pairing-kahler}
H(u,v):=Q(Cu,\bar{v})
\end{eqnarray}
defines a positive definite hermitian 
pairing on \(\mathrm{H}^{n}_{\mathrm{prim}}(X;\mathbb{C})\).
Here \(C\) is the Weil operator: 
\begin{eqnarray*}
Cu=\sqrt{-1}^{p-q}u~\hspace{0.5cm}\mbox{for}~u\in\mathrm{H}^{p,q}(X).
\end{eqnarray*}

As observed by Bogomolov and Tian, the K\"{a}hler potential of \(G\)
is given by the formula
\begin{eqnarray*}
-\log \tilde{Q}(\Omega_{s},\bar{\Omega}_{s}) = -\log \sqrt{-1}^{n}\left(\int_{\mathcal{X}_{s}}
\Omega_{s}\wedge\bar{\Omega}_{s}\right)
\end{eqnarray*}
and the K\"{a}hler form for \(G\) can be calculated by
\begin{eqnarray*}
-\frac{\sqrt{-1}}{2}\partial\bar{\partial}\log \tilde{Q}(\Omega_{s},\bar{\Omega}_{s})
=\frac{\sqrt{-1}}{2}
\sum_{i,j}-\partial_{t_{i}}\partial_{t_{\bar{j}}}
\log \tilde{Q}(\Omega_{s},\bar{\Omega}_{s})\mathrm{d}t_{i}\wedge\mathrm{d}\bar{t}_{j}.
\end{eqnarray*}

In other words, the Weil--Petersson metric is nothing but
the pullback of the Fubini--Study metric induced by \(H\)
via the local Torelli immersion
\begin{eqnarray*}
S\to \mathbf{P}\mathrm{H}^{n}_{\mathrm{prim}}(X;\mathbb{C}).
\end{eqnarray*}

\begin{remark}
The local universal complex deformation of \(X\) is isomorphic
to an open set \(U\) in \(\mathrm{H}^{1}(X;TX)\). Put
\begin{eqnarray*}
\mathrm{H}^{1}(X;TX)_{[\omega]}:=\{[\theta]\in\mathrm{H}^{1}(X;TX)~|~
[\theta\lrcorner\omega]=0~\mbox{in}~\mathrm{H}^{0,2}(X)\}
\end{eqnarray*}
and set \(U_{[\omega]}:=U\cap \mathrm{H}^{1}(X;TX)_{[\omega]}\).
Then \(U_{[\omega]}\) consists of local deformations of \(X\)
preserving the polarization \([\omega]\in\mathrm{H}^{2}(X;\mathbb{Z})\).

If \(X\) is a strict Calabi--Yau, 
i.e.~\(\mathrm{H}^{d}(X;\mathcal{O}_{X})=0\) for \(0<d<n\), then
we particularly have 
\begin{eqnarray*}
\mathrm{H}^{1}(X;TX)_{[\omega]} = \mathrm{H}^{1}(X;TX)
\end{eqnarray*}
when \(n\ge 3\). Moreover, we have \(\mathrm{H}^{1}(X;TX)\cong
\mathrm{H}^{1}(X;\Omega_{X}^{n-1})=\mathrm{H}^{1}_{\mathrm{prim}}(X;\Omega_{X}^{n-1})\), 
because
\begin{eqnarray*}
([\theta]\lrcorner~\Omega)\wedge\omega = \pm \Omega\wedge ([\theta]\lrcorner~\omega)=0.
\end{eqnarray*}
Said differently, the sesquilinear pairing \(H\) in 
\eqref{eq:sesuilinear-pairing-kahler}
defines a hermitian metric on 
\(\mathrm{H}^{n}(X;\mathbb{C})\).
In this case, thanks to the formula for the 
K\"{a}hler potential, the Weil--Petersson metric is 
independent of the choice of the polarization \([\omega]\).
\end{remark}

For a non-K\"{a}hler CY \(\partial\bar{\partial}\)-manifold \(X\), we do not have 
the notion of primitive cohomology and even in the strict case, we do not know
whether or not the second Hodge--Riemann bilinear relation holds on 
\(\mathrm{H}^{n-1,1}(X)\). However, we can still consider the ``potential function''
\begin{eqnarray*}
-\partial\bar{\partial}\log \tilde{Q}(\Omega_{s},\bar{\Omega}_{s})
\end{eqnarray*} 
and consider the associated two form
\begin{eqnarray}
\label{eq:kahler-potential}
-\frac{\sqrt{-1}}{2}\partial\bar{\partial}\log \tilde{Q}(\Omega_{s},\bar{\Omega}_{s})
=\frac{\sqrt{-1}}{2}
\sum_{i,j}-\partial_{t_{i}}\partial_{\bar{t}_{j}}
\log \tilde{Q}(\Omega_{s},\bar{\Omega}_{s})\mathrm{d}t_{i}\wedge\mathrm{d}\bar{t}_{j}.
\end{eqnarray}
If the matrix \((-\partial_{t_{i}}\partial_{\bar{t}_{j}}
\log \tilde{Q}(\Omega_{s},\bar{\Omega}_{s}))_{i,j=1}^{n}\)
on the right hand side of \eqref{eq:kahler-potential}
happens to be semi-positive definite, then we obtain a
(pseudo) metric near \(s\in S\).
In this case, following 
\cite{2019-Popovici-holomorphic-deformations-of-balanced-calabi-yau-d-dbar-manifolds},
we will call it the \emph{period-map metric}.

The purpose of this section is to
study the positivity of the matrix 
\((-\partial_{t_{i}}\partial_{\bar{t}_{j}}
\log \tilde{Q}(\Omega_{s},\bar{\Omega}_{s}))_{i,j=1}^{n}\)
and also give a criterion about 
when the second Hodge--Riemann 
bilinear relation holds on \(\mathrm{H}^{n-1,1}(X)\)
and study the finite distance problem on the moduli space.
We will be mainly interested in the case when \(X\) arises
as a small smoothing of a simple normal crossing variety, i.e.~
\(X\) is a general fiber of a semi-stable one-parameter degeneration 
\(\mathcal{X}\to\Delta\) over a small disc.

For a compact CY \(\partial\bar{\partial}\)-manifold \((X,\omega)\), due to the presence of balanced metrics, one
can define a Weil--Petersson metric on the local universal deformation 
\(S\) of \((X,\omega)\) polarized by a fixed class \([\omega^{n-1}]\).
Given two Dolbeault cohomology classes 
\begin{eqnarray*}
[\theta],~[\eta]
\in\mathrm{H}^{1}(X;TX)_{[\omega^{n-1}]}:=\{[\theta]\in\mathrm{H}^{1}(X;TX)~|~
[\theta\lrcorner\omega^{n-1}]=0~\mbox{in}~\mathrm{H}^{n-2,n}(X)\},
\end{eqnarray*}
following \cite{2019-Popovici-holomorphic-deformations-of-balanced-calabi-yau-d-dbar-manifolds}*{\S 5.2}, one can define
an inner product
\begin{eqnarray*}
G_{WP}([\theta],[\eta]):=\frac{\langle\!\langle \theta\lrcorner~\Omega,
\eta\lrcorner~\Omega\rangle\!\rangle_{\omega}}{\tilde{Q}(\Omega,\bar{\Omega})}
\end{eqnarray*}
where the representatives \(\theta\) and \(\eta\)
are choosen such that 
\(\theta\lrcorner~\Omega\) and 
\(\eta\lrcorner~\Omega\) are both \(\omega\)-minimal and \(\mathrm{d}\)-closed
under the Calabi--Yau isomorphism
\begin{eqnarray*}
\lrcorner~\Omega\colon \mathrm{H}^{1}(X;TX)\xrightarrow{\simeq}\mathrm{H}^{n-1,1}(X)
\end{eqnarray*}
and \(\langle\!\langle -,-\rangle\!\rangle_{\omega}\)
is the \(L^{2}\) inner product induced by \(\omega\).
This induces a metric on \(S\); it is called 
a Weil--Petersson metric in \cite{2019-Popovici-holomorphic-deformations-of-balanced-calabi-yau-d-dbar-manifolds}.

From now on, we assume Hypothesis \ref{assumption:isomorphism}
is valid; that is, 
the nilpotent operator \(N\) induces an isomorphism
between relevant quotients of
the monodromy weight filtration:
\begin{eqnarray*}
N^{k}\colon 
\mathrm{Gr}^{\mathcal{W}(M)}_{n+k}\mathrm{H}^{n}(X;\mathbb{C})\to 
\mathrm{Gr}^{\mathcal{W}(M)}_{n-k}\mathrm{H}^{n}(X;\mathbb{C})
\end{eqnarray*}
is an \emph{isomorphism} for each \(0\le k\le n\).

We also remark that the Hypothesis \ref{assumption:isomorphism}
implies that the monodromy filtration \(\mathcal{W}(M)\) on 
\(\mathrm{H}^{n}(X;\mathbb{C})\)
is identical to the filtration induced by the nilpotent operator \(N\).

\subsection{An infinite distance criterion via Hodge theory}
Let \(f\colon \mathcal{X}\to\Delta\) be a one-parameter 
degeneration of CY \(\partial\bar{\partial}\)-manifolds.
We also assume that \(f\) is a semi-stable model
and every irreducible component in \(E=f^{-1}(0)\) is
a \(\partial\bar{\partial}\)-manifold.
There exists a holomorphic function \(\mathbf{a}(t)\) such that
the multi-valued function (considered as a function on \(\Delta^{\ast}\))
\begin{eqnarray*}
\Omega_{z}:=e^{zN}\mathbf{a}(t)
\end{eqnarray*}
is a local section of the bundle \(\mathcal{F}^{n}\) over \(\Delta^{\ast}\).
The metric two-form for the 
period-map ``metric'' is given by the formula
\begin{eqnarray*}
-\frac{\sqrt{-1}}{2}
\partial\bar{\partial}\log \tilde{Q}(e^{zN}\mathbf{a}(t),\overline{e^{zN}\mathbf{a}(t)}).
\end{eqnarray*}
Here \(\tilde{Q}(-,-):=(\sqrt{-1})^{n}Q(-,-)\) and \(Q(-,-)\) is the topological pairing
(cf.~\eqref{eq:topological-pairing})
on the relevant fiber. Moreover, since \(Q(T\alpha,T\beta)=Q(\alpha,\beta)\),
the nilpotent operator \(N\) is an infinitesimal isometry with respect to \(Q\),
i.e.~\(Q(N\alpha,\beta)+Q(\alpha,N\beta)=0\). In particular, since
\begin{eqnarray*}
0=Q(\mathcal{F}^{n}_{t},\mathcal{F}^{1}_{t})=
Q(e^{-zN}\mathcal{F}^{n}_{t},e^{-zN}\mathcal{F}^{1}_{t}),
\end{eqnarray*}
by taking the limit \(\operatorname{Im}z\to\infty\), we have
\(Q(\mathcal{F}^{n}_{\mathrm{lim}},\mathcal{F}^{1}_{\mathrm{lim}})=0\).
Since \(X\) is Calabi--Yau, we have \(\dim\mathcal{F}^{n}_{\mathrm{lim}}=1\).
The non-degeneracy of \(Q(-,-)\) implies that
\(Q(\mathcal{F}^{n}_{\mathrm{lim}},\overline{\mathcal{F}^{n}}_{\mathrm{lim}})\ne 0\).

Following \cite{1997-Wang-on-the-incompleteness-of-the-weil-petersson-metric-along-degenerations-of-calabi-yau-manifolds}, we now introduce a special class of functions on the
universal cover \(\mathfrak{h}\to\Delta^{\ast}\), \(z\mapsto t=e^{2\pi\sqrt{-1}z}\).
Let us write \(z=x+\sqrt{-1}y\). Then \(t=e^{2\pi \sqrt{-1}x-2\pi y}\).
The function \(t\) has the property that its derivatives has exponential decay 
as \(y\to\infty\). More specifically, for each 
\((r,s)\in(\mathbb{Z}_{\ge 0})^{2}\)
\begin{eqnarray*}
|\partial_{x}^{r}\partial_{y}^{s}(t)|\le C(r,s)\cdot e^{-\pi y}
\end{eqnarray*}
for some positive constant \(C(r,s)\) and \(y\gg 0\).
\begin{notation}
We use the symbol \(\mathbf{h}\) to denote the class of 
smooth functions on \(\mathfrak{h}\)
having the property as above, i.e.~all of its derivatives has exponential decay 
as \(y\to\infty\). The notation \(f\in\mathbf{h}\) means
\(f\) is such a function. It is clear that \(\mathbf{h}\) is closed under
addition and multiplication.
\end{notation}

Consider the power series expansion
\begin{eqnarray*}
\mathbf{a}(t)=a_{0}+a_{1}t+\cdots.
\end{eqnarray*}
 Let 
\(d=\min\{k\in\mathbb{N}~|~N^{k+1}a_{0}=0\}\).

\begin{lemma}
We have
\begin{eqnarray*}
\tilde{Q}(e^{zN}\mathbf{a}(t),\overline{e^{zN}\mathbf{a}(t)})=p(y)+h
\end{eqnarray*}
where \(p(y)\) is a polynomial of degree \(d\) and \(h\in\mathbf{h}\).
Moreover, the leading coefficient of \(p(y)\) is given by
\(\tilde{Q}(a_{0},N^{d}\bar{a_{0}})\ne 0\).
\end{lemma}
\begin{proof}
Note that
\begin{align*}
\tilde{Q}(e^{zN}\mathbf{a}(t),\overline{e^{zN}\mathbf{a}(t)})
&=\tilde{Q}(e^{(z-\bar{z})N}\mathbf{a}(t),\overline{\mathbf{a}(t)})=
\tilde{Q}(e^{2\sqrt{-1}y N}\mathbf{a}(t),\overline{\mathbf{a}(t)}).
\end{align*}
Now the function 
\begin{align*}
y^{m}t^{l}\bar{t}^{k}&=
y^{m}e^{2\pi\sqrt{-1}lz}e^{-2\pi\sqrt{-1}k\bar{z}}\\
&=y^{m}e^{2\pi\sqrt{-1}lx-2\pi ny}e^{-2\pi\sqrt{-1}kx-2\pi ky}
\end{align*} 
belongs to \(\mathbf{h}\) for any \(m\ge 0\) 
as long as \((l,k)\ne (0,0)\). 
Therefore, using the power series expansion,
we obtain 
\begin{align*}
\tilde{Q}(e^{2\sqrt{-1}y N}\mathbf{a}(t),\overline{\mathbf{a}(t)}) = 
\tilde{Q}(e^{2\sqrt{-1}y N}a_{0},\overline{a}_{0}) + h = p(y) + h
\end{align*}
as desired. The polynomial \(p(y)\) has degree \(d\) and the 
leading coefficient is clearly equal to 
\(\tilde{Q}(a_{0},N^{d}\bar{a_{0}})\); this is non-zero since
\(a_{0}\in\mathcal{F}^{n}_{\mathrm{lim}}\),
\(N^{d}\bar{a}_{0}\notin \mathcal{F}^{1}_{\mathrm{lim}}\), and 
\(\dim \mathcal{F}^{0}_{\mathrm{lim}}\slash \mathcal{F}^{1}_{\mathrm{lim}}=1\).
\end{proof}

\begin{proposition}
If \(d>0\), we have
\begin{eqnarray*}
-\partial_{z}\partial_{\bar{z}}\log
\tilde{Q}(e^{zN}\mathbf{a}(t),\overline{e^{zN}\mathbf{a}(t)}) = \frac{d}{y^{2}} +h.
\end{eqnarray*}
In particular, this implies that the 
period-map ``metric'' is really a metric for \(y\gg 0\).
\end{proposition}
\begin{proof}
Let us write 
\begin{eqnarray*}
\tilde{Q}(e^{zN}\mathbf{a}(t),\overline{e^{zN}\mathbf{a}(t)})=p(y)+h.
\end{eqnarray*}
Then for any non-negative integer \(d\), we have
\begin{align*}
-\partial_{z}\partial_{\bar{z}}\log
\tilde{Q}(e^{zN}\mathbf{a}(t),\overline{e^{zN}\mathbf{a}(t)})
&=-\left(\frac{\partial^{2}}{\partial x^{2}}+\frac{\partial^{2}}{\partial y^{2}}\right)
\log(p(y)+h)\\
&=\frac{d}{y^{2}}+h.
\end{align*}
Here we slighly abuse the notation; we use the same \(h\)
to denote different functions in \(\mathbf{h}\).
Since \(h\) decays exponentially as \(y\to\infty\), we conclude that for \(d>0\)
\begin{eqnarray*}
\frac{d}{y^{2}}+h > 0,~\mbox{for}~y\gg 0.
\end{eqnarray*}
This shows that \(-\partial_{z}\partial_{\bar{z}}\log
\tilde{Q}(e^{zN}\mathbf{a}(t),\overline{e^{zN}\mathbf{a}(t)})\)
is positive when \(y\gg 0\).
\end{proof}

Because of the proposition, it makes
sense to talk about the distance
near \(0\in\Delta\).
In fact, in the present situation, we
have
\begin{theorem}
\label{thm_infinite-distance}
If \(d>0\), then \(0\in\Delta\) has infinite distance w.r.t.~the period-map metric.
\end{theorem}
\begin{proof}
By the proceeding proposition, for any 
curve \(\gamma\) in \(\mathfrak{h}\) towards \(\infty\), we see that 
\begin{eqnarray*}
\ell(\gamma)\ge \int_{y=c}^{\infty} \sqrt{\frac{d-\epsilon}{y^{2}}}~\mathrm{d}y=
\int_{y=c}^{\infty} \frac{\sqrt{d-r}}{y}~\mathrm{d}y = \infty.
\end{eqnarray*}
\end{proof}

\begin{theorem}
\label{thm_infinite-distance-wp}
For a one-parameter degeneration \(\mathcal{X}\to\Delta\) of compact CY \(\partial\bar{\partial}\)-manifolds,
if \(d>0\), then \(0\in\Delta\) has infinite distance w.r.t.~\(G_{WP}\).
\end{theorem}
\begin{proof}
This follows from the proceeding theorem and the comparison 
\cite{2019-Popovici-holomorphic-deformations-of-balanced-calabi-yau-d-dbar-manifolds}*{Corollary 5.11}.
\end{proof}

\subsection{A finite distance criterion via Hodge theory and polarization}
For \(d=0\). Then we have \(Na_{0}=0\) and 
\(\tilde{Q}(e^{zN}\mathbf{a}(t),\overline{e^{zN}\mathbf{a}(t)}) \in \mathbf{h}\). Therefore,
\begin{eqnarray*}
\partial_{z}\partial_{\bar{z}}\log
\tilde{Q}(e^{zN}\mathbf{a}(t),\overline{e^{zN}\mathbf{a}(t)}) \in \mathbf{h}.
\end{eqnarray*}
However, we do not know whether or not this is truly a metric, that is, if
\begin{eqnarray*}
-\partial_{z}\partial_{\bar{z}}\log\tilde{Q}(e^{zN}\mathbf{a}(t),\overline{e^{zN}\mathbf{a}(t)})
\end{eqnarray*}
is positive
(even non-negative).
From the calculation in \cite{2019-Popovici-holomorphic-deformations-of-balanced-calabi-yau-d-dbar-manifolds}*{\S 5.3},
we learn that 
\begin{eqnarray}
\label{eq:polarization-positive}
-\partial_{z}\partial_{\bar{z}}\log
\tilde{Q}(e^{zN}\mathbf{a}(t),\overline{e^{zN}\mathbf{a}(t)})>0 \Longleftrightarrow
-\tilde{Q}(\theta\lrcorner~\Omega_{z},\overline{\theta\lrcorner~\Omega_{z}})>0.
\end{eqnarray}
Here \(\theta\) is a representative of 
the Dolbeault cohomology class 
\([\theta]\in\mathrm{H}^{0,1}(\mathcal{X}_{t},T\mathcal{X}_{t})\)
associated to the deformation of \(\mathcal{X}_{t}\) in the given family 
\(f\colon\mathcal{X}\to\Delta\) 
such that the interior product 
\(\theta\lrcorner~\Omega_{z}\) is a \(\mathrm{d}\)-closed \((n-1,1)\)-form.
The right hand side in \eqref{eq:polarization-positive} is valid if
the second Hodge--Riemann bilinear relation
on \(\mathrm{H}^{n-1,1}(\mathcal{X}_{t};\mathbb{C})\) holds.
Assuming this, we immediately have the following.
\begin{theorem}
\label{thm:finite-distance}
If \(d=0\), then \(0\in \Delta\) has finite distance with respect to
the period-map metric.
\end{theorem}
\begin{proof}
It suffices to find a curve having finite distance.
One can use the curve \(\gamma(t):=(c,t+R)\) in \(\mathfrak{h}\)
for a fixed constant \(c\) and \(R>0\). Indeed, we can compute
\begin{align*}
\ell(\gamma)&=\int_{y=R}^{\infty} \sqrt{-\partial_{z}\partial_{\bar{z}}\log
\tilde{Q}(e^{zN}\mathbf{a}(t),\overline{e^{zN}\mathbf{a}(t)})}~|\mathrm{d}z|\\
&=\int_{y=R}^{\infty} \sqrt{-\partial_{z}\partial_{\bar{z}}\log
\tilde{Q}(e^{zN}\mathbf{a}(t),\overline{e^{zN}\mathbf{a}(t)})}~\mathrm{d}y\\
&\le \int_{y=R}^{\infty}\varepsilon e^{-\delta y/2}~\mathrm{d}y <\infty.
\end{align*}
Hence \(0\in\Delta\) has finite distance with respect to 
the period-map metric.
\end{proof}
In the next section, we shall focus on the situation \(n=3\) and discuss when
the second Hodge--Riemann bilinear relation holds on 
\(\mathrm{H}^{2,1}(\mathcal{X}_{t};\mathbb{C})\)
for finite distance degenerations.

\section{The second
Hodge--Riemann bilinear relation for finite distance degenerations}

In this section, Hypothesis \ref{assumption:isomorphism} remains assumed.
In what follows, we restrict ourselves to the threefold case, 
i.e.~\(n=3\).
Put \(X=\mathcal{X}_{t}\).
We also assume that \(X\) is CY in the strict sense, i.e.~
\begin{eqnarray*}
\mathrm{H}^{1}(X;\mathcal{O}_{X})
=\mathrm{H}^{2}(X;\mathcal{O}_{X})=0.
\end{eqnarray*} 
This implies 
\(\mathrm{H}^{1}(X;\mathbb{C})\cong \mathrm{H}^{5}(X;\mathbb{C})=0\).
For instance, these conditions are fulfilled if \(X\) is simply connected.
The strictness also implies 
\(\mathrm{H}^{2}(X;\mathbb{C})=\mathrm{H}^{1,1}(X)\) and
the sequence
\begin{eqnarray}
\label{eq:E11}
E^{0,1}_{1}=\mathrm{H}^{1}(E(1);\mathbb{C})\to
E^{1,1}_{1}=\mathrm{H}^{1}(E(2);\mathbb{C})\to 
E^{2,1}_{1}=\mathrm{H}^{1}(E(3);\mathbb{C})
\end{eqnarray}
is exact at the middle.

We remark that for \(\mathrm{H}^{3}(X;\mathbb{C})\)
and any \(r\), the quotient
\(\mathrm{Gr}_{3+r}^{\mathcal{W}(M)}\mathrm{H}^{3}(X;\mathbb{C})\)
is equal to the cohomology of the complex
\begin{eqnarray}
\label{eq:w-spectral-sq}
E^{-r-1,3+r}_{1}=\mathbf{H}^{2}(\mathrm{Gr}_{r+1}^{\mathcal{W}(M)})\to 
E^{-r,3+r}_{1}=\mathbf{H}^{3}(\mathrm{Gr}_{r}^{\mathcal{W}(M)})\to 
E^{-r+1,3+r}_{1}=\mathbf{H}^{4}(\mathrm{Gr}_{r-1}^{\mathcal{W}(M)})
\end{eqnarray}
where the corresponding morphisms are 
both alternating Gysin maps, both alternating restriction maps, or
an alternating Gysin map composed with an alternating restriction map.

To discuss positivity, we recall a result of
Fujisawa. In \cite{2014-Fujisawa-polarization-on-limiting-mixed-hodge-structures}*{\S7},
Fujisawa 
defined for a log deformation \(E\to\ast\) a \emph{trace morphism}
\begin{eqnarray*}
\operatorname{Tr}\colon \mathbf{H}^{2n}(E;\mathcal{K}_{\mathbb{C}})\to \mathbb{C}
\end{eqnarray*}
and introduced a pairing on \(\mathbf{H}^{q}(E;\mathcal{K}_{\mathbb{C}})\otimes
\mathbf{H}^{2n-q}(E;\mathcal{K}_{\mathbb{C}})\) by means of a 
\emph{weak cohomological mixed Hodge complex \(\mathcal{K}\)} on \(E\).
Here \(n\) is the complex dimension of \(E\).
The hypercohomology of \(\mathcal{K}_{\mathbb{C}}\) also computes the cohomology of 
the nearby fiber. However,
unlike the relative
logarithmic de Rham complex, 
the complex \(\mathcal{K}\) carries a natural 
multiplicative structure and hence together with the trace morphism, 
it induces a pairing between
their cohomology groups.

Specializing to \(q=n\) and noting that \(\mathbf{H}^{n}
(E;\mathcal{K}_{\mathbb{C}})\cong\mathrm{H}^{n}(X;\mathbb{C})\), we obtain a pairing 
\begin{eqnarray*}
\langle -,-\rangle \colon \mathrm{H}^{n}(X;\mathbb{C})\otimes
\mathrm{H}^{n}(X;\mathbb{C})\to\mathbb{C}.
\end{eqnarray*}
The upshot is that the pairing descends to 
a pairing between quotients 
\begin{eqnarray*}
\langle-,-\rangle \colon \mathrm{Gr}_{n+r}^{\mathcal{W}(M)}\mathrm{H}^{n}(X;\mathbb{C})
\otimes \mathrm{Gr}_{n-r}^{\mathcal{W}(M)}\mathrm{H}^{n}(X;\mathbb{C})\to\mathbb{C}
\end{eqnarray*}
and it 
is given by a sum (up to a factor) of the usual cohomological pairing on \(E(k)\)'s.
For the later use, let us give the precise formula for the case \(n=3\) and \(r=0,1\)
and the explanation will be given later in a moment.

\begin{situation}
For \(r=0\), the above sequence \eqref{eq:w-spectral-sq} becomes
\begin{eqnarray}
\label{eq:E03}
\mathrm{H}^{1}(E(2);\mathbb{C})\xrightarrow{a} \mathrm{H}^{3}(E(1);\mathbb{C})\oplus\mathrm{H}^{1}(E(3);\mathbb{C})\xrightarrow{b}\mathrm{H}^{3}(E(2);\mathbb{C})
\end{eqnarray}
whose cohomology gives \(\mathrm{Gr}_{3}^{\mathcal{W}(M)}\mathrm{H}^{3}(X;\mathbb{C})\).
For \([\alpha,\beta]\) and \([\gamma,\delta]\in
\mathrm{Gr}_{3}^{\mathcal{W}(M)}\mathrm{H}^{3}(X;\mathbb{C})\), 
the pairing is defined as
\begin{equation}
\label{eq:pairing-gr33}
\langle [\alpha,\beta],[\gamma,\delta]\rangle 
= (2\pi\sqrt{-1})^{3}\left(\frac{(-1)^{3}}{(2\pi\sqrt{-1})^{3}}\int_{E(1)} \alpha\cup\gamma
+\frac{(-1)^{3}}{(2\pi\sqrt{-1})}
\int_{E(3)} \beta\cup\delta\right)
\end{equation}
where \((\bullet,\bullet)\) is a representative of 
\([\bullet,\bullet]\) inside \(\operatorname{ker}(b)\).
Observe that the pairing is well-defined. If 
\((\alpha,\beta)\in\operatorname{Im}(a)\), then 
\((\alpha,\beta)=(\phi(x),\psi(x))\)
for some \(x\in\mathrm{H}^{1}(E(2);\mathbb{C})\).
(See \S\ref{subsec:st-limiting-mhs} for the notations.) 
On one hand, we calculate
\begin{eqnarray*}
\int_{E(1)} \alpha\cup\gamma=
\int_{E(1)} \phi(x)\cup\gamma=
-(2\pi\sqrt{-1})\int_{E(2)} x\cup\psi(\gamma).
\end{eqnarray*}
On the other hand, we have
\begin{eqnarray*}
\int_{E(3)} \beta\cup\delta=
\int_{E(3)} \psi(x)\cup\delta=
-\frac{1}{2\pi\sqrt{-1}}
\int_{E(2)}x\cup \phi(\delta).
\end{eqnarray*}
The equation \eqref{eq:pairing-gr33} becomes
\begin{eqnarray*}
(2\pi\sqrt{-1})^{3}\left(-\frac{(-1)^{3}}{(2\pi\sqrt{-1})^{2}}\int_{E(2)} 
x\cup\psi(\gamma)
-\frac{(-1)^{3}}{(2\pi\sqrt{-1})^{2}}
\int_{E(3)} x\cup\phi(\delta)\right)=0
\end{eqnarray*}
since \(\psi(\gamma)+\phi(\delta)=0\).
The remaining cases can be checked in a similar fashion.
\end{situation}

\begin{situation}
For \(r=1\), the quotient \(\mathrm{Gr}_{4}^{\mathcal{W}(M)}
\mathrm{H}^{3}(X;\mathbb{C})\) is given by
the cohomology of
\begin{eqnarray*}
\mathrm{H}^{0}(E(3);\mathbb{C})\to 
\mathrm{H}^{2}(E(2);\mathbb{C})\oplus\mathrm{H}^{0}(E(4);\mathbb{C})\to 
\mathrm{H}^{4}(E(1);\mathbb{C})\oplus\mathrm{H}^{2}(E(3);\mathbb{C})
\end{eqnarray*}
while for \(r=-1\), the quotient \(\mathrm{Gr}_{2}^{\mathcal{W}(M)}
\mathrm{H}^{3}(X;\mathbb{C})\) is given by
the cohomology of
\begin{eqnarray*} 
\mathrm{H}^{2}(E(1);\mathbb{C})\oplus\mathrm{H}^{0}(E(3);\mathbb{C})\to 
\mathrm{H}^{2}(E(2);\mathbb{C})\oplus\mathrm{H}^{0}(E(4);\mathbb{C})\to 
\mathrm{H}^{2}(E(3);\mathbb{C}).
\end{eqnarray*}
For \([\alpha,\beta]\in
\mathrm{Gr}_{4}^{\mathcal{W}(M)}\mathrm{H}^{3}(X;\mathbb{C})\)
and \([\gamma,\delta]\in
\mathrm{Gr}_{2}^{\mathcal{W}(M)}\mathrm{H}^{3}(X;\mathbb{C})\)
the pairing is given by
\begin{equation}
\label{eq:pairing-gr24}
\langle [\alpha,\beta],[\gamma,\delta]\rangle 
= (2\pi\sqrt{-1})^{3}\left(\frac{(-1)^{3}}{(2\pi\sqrt{-1})^{2}}
\int_{E(2)} \alpha\cup\gamma
+\frac{(-1)^{3}}{(2\pi\sqrt{-1})^{0}}
\int_{E(4)} \beta\cup\delta\right)
\end{equation}
where \((\bullet,\bullet)\) is a lifting of 
\([\bullet,\bullet]\) in the corresponding vector spaces.
One can check again that this is well-defined. 
\end{situation}

Now recall that the family is a finite distance degeneration, i.e.~\(Na_{0}=0\).
This in turns implies \(N^{2}=0\) on \(\mathrm{H}^{3}(X;\mathbb{C})\)
and consequently \(\mathrm{Gr}_{1}^{\mathcal{W}(M)}\mathrm{H}^{3}(X;\mathbb{C})=0\); namely~the
following sequence is exact
\begin{eqnarray}
\label{eq:E21}
E_{1}^{1,1}=\mathrm{H}^{1}(E(2);\mathbb{C})\to
E^{2,1}_{1}=\mathrm{H}^{1}(E(3);\mathbb{C})\to E^{3,1}_{1}=\mathrm{H}^{1}(E(4);\mathbb{C})=0.
\end{eqnarray}

Recall that a one-parameter semi-stable degeneration \(f\colon\mathcal{X}\to\Delta\) 
satisfying Hypothesis \ref{assumption:com-kahler}
is called a \emph{one-parameter K\"{a}hler degeneration} in \cite{2008-Peters-Steenbrink-mixed-Hodge-structures}. We remark that
Hypothesis \ref{assumption:com-kahler} also implies
Hypothesis \ref{assumption:isomorphism} by 
\cite{2008-Peters-Steenbrink-mixed-Hodge-structures}*{Theorem 11.40}.

Let us explain how the pairings \eqref{eq:pairing-gr33} and
\eqref{eq:pairing-gr24} are related to the topological pairing on \(X\).
Under Hypothesis \ref{assumption:com-kahler}, we have the Clemens--Schmid exact 
sequence (cf.~\cite{2008-Peters-Steenbrink-mixed-Hodge-structures}*{Corollary 11.44})
\begin{eqnarray*}
\cdots\to\mathrm{H}_{2n+2-m}(E;\mathbb{C})\to
\mathrm{H}^{m}(E;\mathbb{C})\xrightarrow{i}
\mathrm{H}^{m}(X;\mathbb{C})\xrightarrow{N}
\mathrm{H}^{m}(X;\mathbb{C})\to\mathrm{H}_{2n-m}(E;\mathbb{C})\to\cdots.
\end{eqnarray*}
We can put a mixed Hodge structure on 
\(\mathrm{H}^{k}(E;\mathbb{C})\) via the Mayer--Vietoris resolution and
we shall denote by \(W\) the corresponding weight filtration.
This is an exact sequence of mixed Hodge structures
whose morphisms are of type \((n+1,n+1)\), \((0,0)\), \((-1,-1)\),
and \((-n,-n)\) respectively. The map \(N\) is the nilpotent operator while
\(i\) is the map induced from the inclusion \(X\subset \mathcal{X}\) and
the deformation retraction \(E\to \mathcal{X}\).

Specializing to \(n=m=3\), we obtain
\begin{eqnarray*}
0\to\mathrm{H}_{5}(E;\mathbb{C})\to
\mathrm{H}^{3}(E;\mathbb{C})\xrightarrow{i}
\mathrm{H}^{3}(X;\mathbb{C})\xrightarrow{N}
\mathrm{H}^{3}(X;\mathbb{C})\to\mathrm{H}_{3}(E;\mathbb{C})\to\cdots.
\end{eqnarray*}
Look at the graded piece
\begin{eqnarray*}
0\to\mathrm{Gr}_{-5}^{W}\mathrm{H}_{5}(E;\mathbb{C})
\to \mathrm{Gr}_{3}^{W}\mathrm{H}^{3}(E;\mathbb{C})
\to \mathrm{Gr}_{3}^{\mathcal{W}(M)}\mathrm{H}^{3}(X;\mathbb{C})\xrightarrow{N} 0.
\end{eqnarray*}
Here the last \(0\) comes from the assumption that \(\mathcal{X}\to \Delta\)
has finite distance. 
Consider the commutative diagram
\begin{eqnarray*}
\begin{tikzcd}
& &\mathrm{H}^{3}(E(1);\mathbb{C})\ar[r,"\phi"]\ar[d]&\mathrm{H}^{3}(E(2);\mathbb{C})\ar[d]\\
&\mathrm{H}^{1}(E(2);\mathbb{C})\ar[r,"a"]
&\mathrm{H}^{3}(E(1);\mathbb{C})\oplus\mathrm{H}^{1}(E(3);\mathbb{C})\ar[r,"b"]
&\mathrm{H}^{3}(E(2);\mathbb{C}).
\end{tikzcd}
\end{eqnarray*}
We then obtain a map
\begin{eqnarray*}
\operatorname{ker}(\phi)\to \operatorname{ker}(b)\to 
\operatorname{ker}(b)\slash \operatorname{im}(a)
\end{eqnarray*}
which gives the surjection in the Clemens--Schmid exact sequence
\begin{eqnarray*}
\mathrm{Gr}_{3}^{W}\mathrm{H}^{3}(E;\mathbb{C})
\to \mathrm{Gr}_{3}^{\mathcal{W}(M)}\mathrm{H}^{3}(X;\mathbb{C}).
\end{eqnarray*}

We have a pairing \eqref{eq:pairing-gr33} on \(\operatorname{ker}(b)\).
From the proceeding discussion, for any \([\alpha_{i},\beta_{i}]\in
\operatorname{ker}(b)\), we can find an element \(\xi_{i}\in\operatorname{ker}(\phi)\)
such that \([\xi_{i},0]=[\alpha_{i},\beta_{i}]\) by the surjectivity.
It follows that
\begin{align}
\begin{split}
\langle [\alpha_{1},\beta_{1}],[\alpha_{2},\beta_{2}]\rangle =
\langle [\xi_{1},0],[\xi_{2},0]\rangle = (-1)^{3}\int_{E(1)}\xi_{1}\cup \xi_{2}
=(-1)^{3}\int_{E}\xi_{1}\cup \xi_{2}.
\end{split}
\end{align}
In the last integral, we think of \(\xi_{i}\) as an element in \(\mathrm{H}^{3}(E;\mathbb{C})\)
and use the fact that \(\mathrm{H}_{6}(E;\mathbb{Z})=\mathrm{H}_{6}(E(1);\mathbb{Z})\).
Together with the deformation retraction \(\mathcal{X}\supset E\) and
noticing that \(X\) and \(E\) are homologuous,
the last integral can be identified with the topological pairing 
on \(X\), i.e.,
\begin{eqnarray*}
(-1)^{3}\int_{E}\xi_{1}\cup \xi_{2}=(-1)^{3}\int_{X} i(\xi_{1})\cup i(\xi_{2}),
\end{eqnarray*}
and this gives \eqref{eq:pairing-gr33} up to twist.

Similarly, we have an isomorphism
\begin{eqnarray*}
\mathrm{Gr}_{2}^{W}\mathrm{H}^{3}(E;\mathbb{C})
\to \mathrm{Gr}_{2}^{\mathcal{W}(M)}\mathrm{H}^{3}(X;\mathbb{C})
\end{eqnarray*}
from the Clemens--Schmid exact sequence; this is deduced
from the commutative diagram
\begin{eqnarray*} 
\begin{tikzcd}
&\mathrm{H}^{2}(E(1);\mathbb{C})\ar[r,"\phi"]\ar[d]
&\mathrm{H}^{2}(E(2);\mathbb{C})\ar[r]\ar[d]
&\mathrm{H}^{2}(E(3);\mathbb{C})\ar[d,equal]\\
&\mathrm{H}^{2}(E(1);\mathbb{C})\oplus\mathrm{H}^{0}(E(3);\mathbb{C})\ar[r,"c"]
&\mathrm{H}^{2}(E(2);\mathbb{C})\oplus\mathrm{H}^{0}(E(4);\mathbb{C})\ar[r,"d"]
&\mathrm{H}^{2}(E(3);\mathbb{C}).
\end{tikzcd}
\end{eqnarray*}
We can construct a long exact sequence involving the map \(\phi\) in
the top row as follows.
Consider the short exact sequence
\begin{eqnarray*} 
0\to A^{\bullet}(E(2))[-1]\to A^{\bullet}(E(1))
\oplus A^{\bullet-1}(E(2))\to A^{\bullet}(E(1))\to 0
\end{eqnarray*}
where the second map is given by \(\eta\to (0,\eta)\) while
the third one is the projection. The differential \(D\) on the complex
\(A^{\bullet}(E(1))
\oplus A^{\bullet-1}(E(2))\) is given by the formula
\begin{eqnarray*}
D(\omega,\eta):=(\mathrm{d}\omega, \phi(\omega)-\mathrm{d}\eta).
\end{eqnarray*}
One checks \(D^{2}=0\). The cohomology
group is denoted by \(\mathrm{H}^{k}(E(1);E(2))\).
We have an induced long exact sequence
\begin{eqnarray*}
\cdots\to \mathrm{H}^{k-1}(E(2))\to \mathrm{H}^{k}(E(1);E(2))\to 
\mathrm{H}^{k}(E(1))\xrightarrow{\varepsilon}\mathrm{H}^{k}(E(2))\to\cdots
\end{eqnarray*}
whose connecting homomorphism \(\varepsilon\) is equal to \(\phi\).
We have an injection
\begin{eqnarray*}
\mathrm{Gr}_{2}^{W}\mathrm{H}^{3}(E;\mathbb{C})\hookrightarrow
\mathrm{H}^{2}(E(2))\slash\operatorname{im}(\phi) \hookrightarrow\mathrm{H}^{3}(E(1);E(2)).
\end{eqnarray*}
A diagram chasing shows that under the inclusion, we have
\begin{eqnarray*}
\mathrm{Gr}_{2}^{\mathcal{W}(M)}\mathrm{H}^{3}(X;\mathbb{C})
\cong \mathrm{Gr}_{2}^{W}\mathrm{H}^{3}(E;\mathbb{C})\hookrightarrow
\mathrm{H}^{3}_{\mathrm{c}}(U;\mathbb{C})
\end{eqnarray*}
where \(U=E\setminus E_{\mathrm{sing}}\).
Consider the dual diagram
\begin{eqnarray*}
\begin{tikzcd}
&\mathrm{H}_{2}(E(3);\mathbb{C})\ar[r,"\Psi'"]
&\mathrm{H}_{2}(E(2);\mathbb{C})\ar[r,"\Psi"]
&\mathrm{H}_{2}(E(1);\mathbb{C})\\
&\mathrm{H}^{0}(E(3);\mathbb{C})\ar[u,"\cong"]\ar[r]\ar[d,equal]
&\mathrm{H}^{2}(E(2);\mathbb{C})\ar[u,"\cong"]\ar[r,"\psi"]
&\mathrm{H}^{4}(E(1);\mathbb{C})\ar[u,"\cong"]\\
&\mathrm{H}^{0}(E(3);\mathbb{C})\ar[r]
&\mathrm{H}^{2}(E(2);\mathbb{C})\oplus\mathrm{H}^{0}(E(4);\mathbb{C})\ar[r]\ar[u]
&\mathrm{H}^{4}(E(1);\mathbb{C})\oplus\mathrm{H}^{2}(E(3);\mathbb{C}).\ar[u]
\end{tikzcd}
\end{eqnarray*}
Write \(\Psi=(\Psi_{1},\ldots,\Psi_{m})\).
If \((\sigma_{I})\in \mathrm{H}_{2}(E(2);\mathbb{C})\)
belongs to \(\operatorname{ker}(\Psi)\), then
\(\Psi_{j}(\sigma_{I})\) is a coboundary, i.e.~there exists a
\(3\)-chain \(\Gamma_{j}\) such that 
\(\partial\Gamma_{j}=\Psi_{j}(\sigma_{I})\).
Let us denote by \(\mathcal{S}\) the collection of
chains constructed as above, i.e.~
\begin{eqnarray*}
\mathcal{S}=\{\Gamma:=(\Gamma_{1},\ldots,\Gamma_{m})~|~\Gamma_{i}\in \mathrm{H}_{3}(E_{i},\cup_{j\ne i}E_{ij};\mathbb{C}),~\partial\Gamma_{i}=\Psi_{i}(\sigma_{I})~\mbox{for all}~
(\sigma_{I})\in\operatorname{ker}(\Psi)\}.
\end{eqnarray*}
Note that for each \(\Gamma\in\mathcal{S}\) the chains \(\Gamma_{i}\)'s can
be glued into a \(3\)-cycle in \(E\) along their boundaries.

Recall that 
\begin{eqnarray*}
\Psi = \bigoplus_{j=1}^{2} (-1)^{j-1}(\iota_{2,j})_{\ast}\colon 
\mathrm{H}_{q}(E(2);\mathbb{C})\to \mathrm{H}_{q}(E(1);\mathbb{C})
\end{eqnarray*}
is defined on the level of chains.
Combined with Lefschetz duality, it gives rise to a map
\begin{eqnarray*}
\operatorname{ker}(\Psi)\to \mathcal{S}\subseteq
\bigoplus_{i=1}^{m}\mathrm{H}_{3}(E_{i},\cup_{j\ne i}E_{ij};\mathbb{C})\subseteq
\mathrm{H}^{3}(U;\mathbb{C})
\end{eqnarray*}
an induces an injection
\begin{eqnarray*}
\mathrm{Gr}_{2}^{\mathcal{W}(M)}\mathrm{H}^{3}(X;\mathbb{C})\cong\operatorname{ker}(\Psi)
\slash\operatorname{im}(\Psi')\to \mathcal{S}.
\end{eqnarray*}
Consequently, the induced pairing 
\begin{eqnarray*}
\langle -,-\rangle\colon \mathrm{Gr}_{4}^{\mathcal{W}(M)}\mathrm{H}^{3}(X;\mathbb{C})
\times \mathrm{Gr}_{2}^{\mathcal{W}(M)}\mathrm{H}^{3}(X;\mathbb{C}) \to \mathbb{C}
\end{eqnarray*}
can be computed via the usual pairing between
relative cohomology and relative homology, or
the uaual Poincar\'{e} pairing
\(\mathrm{H}^{3}(U)\times\mathrm{H}^{3}_{\mathrm{c}}(U)\to \mathrm{H}^{6}_{\mathrm{c}}(U)
\cong\mathrm{H}_{0}(U)
\cong\mathbb{C}\) 
on \(U\). Here the isomorphism \(\mathrm{H}^{6}_{\mathrm{c}}(U)
\cong\mathbb{C}\) comes from integration over \(U\). 
This gives \eqref{eq:pairing-gr24}
up to twist.
\begin{remark}
We remark that the hermitian pairing
\begin{eqnarray*}
\sqrt{-1}\langle u,N_{\mathrm{st}}\bar{v}\rangle
\end{eqnarray*}
is \emph{negative definite} on 
\(\mathrm{Gr}_{4}^{\mathcal{W}(M)}\mathrm{H}^{3}(X;\mathbb{C})\), as
one can easily derive from the formula \eqref{eq:pairing-gr24}.
Here \(N_{\mathrm{st}}\) is the nilpotent operator 
from Steenbrink's theory; it can be
regarded as the induced map from the canonical map 
\begin{eqnarray*}
\mathcal{D}^{p,q}\to \mathcal{D}^{p+1,q-1}~\mbox{via}~
\omega\mapsto \overline{\omega}.
\end{eqnarray*}
We note that \(N=-N_{\mathrm{st}}\). In fact, recall that 
the trivialization of \(\mathcal{X}_{\infty}\to\mathfrak{h}\)
induces a diffeomorphism
\begin{eqnarray*}
\mathcal{X}_{\infty,0}\xrightarrow{\simeq}\mathcal{X}_{\infty,1}
\end{eqnarray*}
and hence a monodromy operator \(T_{\mathrm{st}}
\colon \mathrm{H}^{k}(\mathcal{X}_{\infty,1};\mathbb{C})
\to \mathrm{H}^{k}(\mathcal{X}_{\infty,0};\mathbb{C})\).
We have \(T=(T_{\mathrm{st}}^{-1})^{\ast}\)
where \(T\) is the monodromy operator described in \S\ref{subsec:deligne-ext}.
Consequently, we infer that the hermitian pairing
\begin{eqnarray*}
\sqrt{-1}\langle u,N\bar{v}\rangle
\end{eqnarray*}
becomes \emph{positive definite}.
\end{remark}

\begin{proposition}
\label{prop:finite-distance-polarized}
Assuming the Hypothesis \ref{assumption:com-kahler}, then the quotient
\(\mathrm{Gr}_{3}^{\mathcal{W}(M)}\mathrm{H}^{3}(X;\mathbb{C})\) is equipped
with a polarized Hodge structure.
\end{proposition}
\begin{proof}
Since \(\mathrm{H}^{5}(X;\mathbb{C})=0\), we have
\(\mathrm{Gr}_{5}^{\mathcal{W}(M)}\mathrm{H}^{5}(X;\mathbb{C})=0\), i.e.~the sequence
\begin{eqnarray*}
\mathrm{H}^{3}(E(2);\mathbb{C})
\to \mathrm{H}^{5}(E(1);\mathbb{C})\to 0
\end{eqnarray*}
is exact at the middle. 

Let \(\omega_{I}\) be the K\"{a}hler form on \(E_{I}\) coming from the
restriction of the universal class
in \(\mathrm{H}^{2}(E;\mathbb{R})\) under Hypothesis \ref{assumption:com-kahler}.
Consider the commutative diagram
\begin{equation}
\begin{tikzcd}
&\mathrm{H}^{1}(E(2);\mathbb{C})
\ar[r,"a"]\ar[d] &\mathrm{H}^{3}(E(1);\mathbb{C})\oplus\mathrm{H}^{1}(E(3);\mathbb{C})\ar[d,"L"]\\
&\mathrm{H}^{3}(E(2);\mathbb{C})
\ar[r] &\mathrm{H}^{5}(E(1);\mathbb{C}).
\end{tikzcd}
\end{equation}
In this diagram, 
\begin{itemize}
\item the top row is from the 
spectral sequence computing \(\mathrm{Gr}_{3}^{\mathcal{W}(M)}\mathrm{H}^{3}(X;\mathbb{C})\);
\item the bottom row is from the 
spectral sequence computing \(\mathrm{Gr}_{5}^{\mathcal{W}(M)}\mathrm{H}^{5}(X;\mathbb{C})\);
\item the vertical maps are give by a product of Lefschetz operators using \(\omega_{I}\).
\end{itemize}
From the discussion above, the bottom arrow is surjective. 
Moreover, the left vertical arrow is an isomorphism.
Thus 
\begin{eqnarray*}
\left.L\right|_{\operatorname{Im}(a)}
\colon \operatorname{Im}(a)\to \mathrm{H}^{5}(E(1);\mathbb{C})
\end{eqnarray*}
is also a surjection. In particular, if 
\begin{eqnarray*}
[\gamma,\delta]\in \mathrm{Gr}_{3}^{\mathcal{W}(M)}\mathrm{H}^{3}(X;\mathbb{C})
\end{eqnarray*}
and \((\gamma,\delta)\in 
\mathrm{H}^{3}(E(1);\mathbb{C})\oplus\mathrm{H}^{1}(E(3);\mathbb{C})\) is any representative,
we can use elements in \(\operatorname{Im}(a)\) to modify \((\gamma,\delta)\)
such that \(L(\gamma,\delta)=0\), i.e.
\begin{eqnarray*}
\gamma\in\mathrm{H}^{3}_{\mathrm{prim}}(E(1);\mathbb{C}).
\end{eqnarray*}
Therefore the quotient
\(\mathrm{Gr}_{3}^{\mathcal{W}(M)}\mathrm{H}^{3}(X;\mathbb{C})\)
carries a polarized Hodge structure.
\end{proof}

The associated
sesquilinear pairing \(\langle Cu,\bar{v}\rangle\) 
(here \(C\) is the Weil operator) 
between the quotients \(\mathrm{Gr}_{4}^{\mathcal{W}(M)}\mathrm{H}^{3}(X;\mathbb{C})\)
and \(\mathrm{Gr}_{2}^{\mathcal{W}(M)}\mathrm{H}^{3}(X;\mathbb{C})\)
is positive definite as
long as the surface \(E(2)\) is compact K\"{a}hler and the statement
\begin{equation}
\label{eq:implication}
``\mbox{\([\alpha,\beta]\in\mathrm{Gr}_{4}^{\mathcal{W}(M)}\mathrm{H}^{3}(X;\mathbb{C})
\)}~\Rightarrow~\alpha\in \mathrm{H}^{2}_{\mathrm{prim}}(E(2);\mathbb{C})"
\tag{\(\ast\)}
\end{equation}
holds. Fortunately, we have

\begin{proposition}
Under the Hypothesis \ref{assumption:com-kahler}, if 
\begin{eqnarray*}
[\alpha,\beta]\in\mathrm{Gr}_{4}^{\mathcal{W}(M)}\mathrm{H}^{3}(X;\mathbb{C})
\end{eqnarray*}
then there exists a representative \((\alpha,\beta)\)
such that \(\alpha\in \mathrm{H}^{2}_{\mathrm{prim}}(E(2);\mathbb{C})\).
\end{proposition}
\begin{proof}
Since \(\mathrm{H}^{5}(X;\mathbb{C})=0\), we have
\(\mathrm{Gr}_{4}^{\mathcal{W}(M)}\mathrm{H}^{5}(X;\mathbb{C})=0\), i.e.~the sequence
\begin{eqnarray*}
\mathrm{H}^{4}(E(1);\mathbb{C})\oplus\mathrm{H}^{2}(E(3);\mathbb{C})
\to \mathrm{H}^{4}(E(2);\mathbb{C})\to 0
\end{eqnarray*}
is exact at the middle. 

Let \(\omega_{I}\) be the K\"{a}hler form on \(E_{I}\) coming from the
restriction of the universal class
in \(\mathrm{H}^{2}(E;\mathbb{R})\) under Hypothesis \ref{assumption:com-kahler}.
Consider the commutative diagram
\begin{equation}
\begin{tikzcd}
&\mathrm{H}^{2}(E(1);\mathbb{C})\oplus\mathrm{H}^{0}(E(3);\mathbb{C})
\ar[r,"a"]\ar[d] &\mathrm{H}^{2}(E(2);\mathbb{C})\oplus\mathrm{H}^{0}(E(4);\mathbb{C})\ar[d,"L"]\\
&\mathrm{H}^{4}(E(1);\mathbb{C})\oplus\mathrm{H}^{2}(E(3);\mathbb{C})
\ar[r] &\mathrm{H}^{4}(E(2);\mathbb{C}).
\end{tikzcd}
\end{equation}
In this diagram, 
\begin{itemize}
\item the top row is from the 
spectral sequence computing \(\mathrm{Gr}_{2}^{\mathcal{W}(M)}\mathrm{H}^{3}(X;\mathbb{C})\);
\item the bottom row is from the 
spectral sequence computing \(\mathrm{Gr}_{4}^{\mathcal{W}(M)}\mathrm{H}^{5}(X;\mathbb{C})\);
\item the vertical maps are give by a product of Lefschetz operators using \(\omega_{I}\).
\end{itemize}
From the discussion above, the bottom arrow is surjective. 
Moreover, the left vertical arrow is an isomorphism.
Thus 
\begin{eqnarray*}
\left.L\right|_{\operatorname{Im}(a)}
\colon \operatorname{Im}(a)\to \mathrm{H}^{4}(E(2);\mathbb{C})
\end{eqnarray*}
is also a surjection. In particular, if 
\begin{eqnarray*}
[\gamma,\delta]\in \mathrm{Gr}_{2}^{\mathcal{W}(M)}\mathrm{H}^{3}(X;\mathbb{C})
\end{eqnarray*}
and \((\gamma,\delta)\in 
\mathrm{H}^{2}(E(2);\mathbb{C})\oplus\mathrm{H}^{0}(E(4);\mathbb{C})\) is any representative,
we can use elements in \(\operatorname{Im}(a)\) to modify \((\gamma,\delta)\)
such that \(L(\gamma,\delta)=0\), i.e.
\begin{eqnarray*}
\gamma\in\mathrm{H}^{2}_{\mathrm{prim}}(E(2);\mathbb{C}).
\end{eqnarray*}
Next, under the Hypothesis \ref{assumption:isomorphism}, that is, the identity map
\begin{equation*}
\begin{tikzcd}[column sep=1em]
&\mathrm{H}^{0}(E(3);\mathbb{C})\ar[r,"c"] 
&\mathrm{H}^{2}(E(2);\mathbb{C})\oplus\mathrm{H}^{0}(E(4);\mathbb{C})\ar[r,"d"]\ar[d,"\mathrm{id}"]
&\mathrm{H}^{4}(E(1);\mathbb{C})\oplus\mathrm{H}^{2}(E(3);\mathbb{C})\\
&\mathrm{H}^{2}(E(1);\mathbb{C})\oplus\mathrm{H}^{0}(E(3);\mathbb{C})\ar[r,"a"]
&\mathrm{H}^{2}(E(2);\mathbb{C})\oplus\mathrm{H}^{0}(E(4);\mathbb{C})\ar[r,"b"]
&\mathrm{H}^{2}(E(3);\mathbb{C})
\end{tikzcd}
\end{equation*}
induces an isomorphism between 
\(\operatorname{ker}(d)\slash\operatorname{im}(c)\)
and \(\operatorname{ker}(b)\slash\operatorname{im}(a)\),
it follows that \(\operatorname{im}(a)\subseteq \operatorname{im}(c)\).
For 
\([\alpha,\beta]\in\mathrm{Gr}_{4}^{\mathcal{W}(M)}\mathrm{H}^{3}(X;\mathbb{C})\), if 
\((\alpha,\beta)\in \mathrm{H}^{2}(E(2);\mathbb{C})\oplus\mathrm{H}^{0}(E(4);\mathbb{C})\)
is a representative, we can use elements in \(\operatorname{im}(a)
\subseteq\operatorname{im}(c)\)
to adjust \((\alpha,\beta)\) so that 
\(\alpha\in \mathrm{H}^{2}_{\mathrm{prim}}(E(2);\mathbb{C})\) as desired.
\end{proof}

Now we can state our main result in this section.
\begin{theorem}
\label{theorem:polarized-hs}
Let \(\mathcal{X}\to\Delta\) be a finite distance and semi-stable degeneration
of smooth CY \(\partial\bar{\partial}\)-threefolds.
Let \(E\) be the central fiber
and \(X\) be a general fiber sufficiently close to \(E\). 
Under Hypothesis \ref{assumption:com-kahler}, 
the Hodge structure on \(\mathrm{H}^{3}(X;\mathbb{C})\) is polarized.
\end{theorem}
The rest of this section is devoted to demonstrate 
Theorem \ref{theorem:polarized-hs}. We shall closely follow Li's
idea in \cite{2022-Li-polarized-hodge-structures-for-clemens-manifolds}.
Consider Deligne's splitting \(I^{p,q}\) for 
Steenbrink's limiting mixed Hodge structure 
on \(\mathrm{H}^{3}(X;\mathbb{C})\) (cf.~Theorem \ref{theorem:steenbrink-mhs}).
Firstly, we construct a basis for each \(I^{p,q}\) and hence
obtain a basis for \(\mathcal{F}_{\mathrm{lim}}^{2}\). 
Secondly, we extend the proceeding basis to become a local frame 
for Deligne's canonical extension. Finally we untwist the local frame by \(e^{zN}\)
to get a (multi-valued) frame for \(\mathcal{F}^{2}_{t}\).
We then check the positivity using the multi-valued frame.

\subsection{A construction of the basis for the canonical extension}

Let the notations be as before. We construct a basis for 
\(\mathcal{F}_{\mathrm{lim}}^{2}\) as follows. Recall that we have
\begin{equation}
W_{3}=W_{2}\oplus \left(\bigoplus_{p=0}^{3} I^{p,3-p}\right).
\end{equation}
Notice that both \(W_{2}\) and \(\bigoplus_{p=0}^{3} I^{p,3-p}\)
are complex vector spaces with real stucture
coming from the complex conjugation on \(H\).
Indeed, by 
\begin{equation}
I^{p,q} = \overline{I^{q,p}} \mod{\bigoplus_{0\le r<p,~0\le s<q} I^{r,s}},
\end{equation}
it follows that \(W_{2} = I^{1,1} = \overline{I^{1,1}}=\overline{W_{2}}\) and 
\(V = \overline{V}\) for \(V:=\bigoplus_{p=0}^{3} I^{p,3-p}\).
Consequently, we can pick a \emph{real} basis
\(\{\alpha_{1},\ldots,\alpha_{m}\}\) for \(W_{2}\) and 
\(\{\beta_{1},\ldots,\beta_{h+1},\beta_{h+2},\ldots,\beta_{2h+2}\}\) 
for \(V\). Now consider \(I^{1,1}\oplus I^{2,2}\). 
This is again a complex vector space with real structure
by \eqref{eq:deligne-splitting} and therefore it also 
admits a real basis. We may extend the basis \(\{\alpha,\ldots,\alpha_{m}\}\)
to become a real basis 
\(\{\alpha_{1},\ldots,\alpha_{m},\gamma_{1},\ldots,\gamma_{m}\}\) 
for \(I^{1,1}\oplus I^{2,2}\). Also, there are complex numbers
\(x_{kl}\in\mathbb{C}\) with \(1\le k,l\le m\) such that
\begin{equation}
\delta_{k}:=\gamma_{k}-\sum_{l=1}^{m} x_{kl}\alpha_{l}\in I^{2,2}.
\end{equation}
Under the Hypothesis \ref{assumption:isomorphism}, we 
may further assume that \(N(\delta_{k})=N(\gamma_{k})=\alpha_{k}\).

\begin{figure}
\includegraphics[scale=0.55]{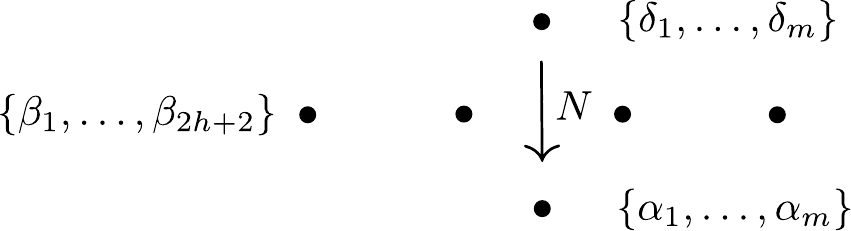}
\caption{The limiting mixed Hodge diamond for the
middle cohomology group \(\mathrm{H}^{3}(X;\mathbb{C})\) and the bases for
the corresponding subspaces \(I^{p,q}\).}
\label{fig:mixed-hodge}
\end{figure}

\subsection*{Notation}
Let us fix the notation for the later use.
\begin{itemize}
\item Let \(\{u_{1},\ldots,u_{h+1}\}\) be
a basis for \(I^{3,0}\oplus I^{2,1}\). We can
write \(u_{p}=\sum_{i=1}^{2h+2} b^{i}_{p} \beta_{i}\).
\item Let \(\{v_{1}=\delta_{1},\ldots,v_{m}=\delta_{m}\}\) be a basis for \(I^{2,2}\). 
\end{itemize}
Since \(e^{-zN}\mathcal{F}^{2}\) is a holomorphic subbundle in 
Deligne's canonical extension, we may extend
\begin{equation}
\{u_{1},\ldots,u_{h+1},v_{1},\ldots,v_{m}\}
\end{equation}
to become a local frame for \(e^{-zN}\mathcal{F}^{2}\)
over \(\Delta\). Explicitly, we can write
\begin{align}
\begin{split}
u_{p}(t) &= \sum_{i=1}^{m} A^{i}_{p}(t) \alpha_{i}+
\sum_{i=1}^{2h+2} B^{i}_{p}(t) \beta_{i} + \sum_{i=1}^{m} C^{i}_{p}(t) \delta_{i},\\
v_{q}(t) &= \sum_{i=1}^{m} D^{i}_{q}(t) \alpha_{i}+
\sum_{i=1}^{2h+2} E^{i}_{q}(t) \beta_{i} + \sum_{i=1}^{m} F^{i}_{q}(t) \delta_{i}.
\end{split}
\end{align}
In the above expression, all the coefficients are \emph{holomorphic}.

In terms of the real basis
\(\{\alpha,\ldots,\alpha_{m},\beta_{1},\ldots,\beta_{2h+2},\gamma_{1},\ldots,\gamma_{m}\}\), 
we can write
\begin{align}
\begin{split}
u_{p}(t) &= \sum_{i=1}^{m} \left(A^{i}_{p}(t)-
\sum_{k=1}^{m}C_{p}^{k}(t)x_{ki}\right) \alpha_{i}+
\sum_{i=1}^{2h+2} B^{i}_{p}(t) \beta_{i} + 
\sum_{i=1}^{m} C^{i}_{p}(t) \gamma_{i},\\
v_{q}(t) &= \sum_{i=1}^{m} \left(D^{i}_{q}(t)-
\sum_{k=1}^{m}F_{q}^{k}(t)x_{ki}\right) \alpha_{i}+
\sum_{i=1}^{2h+2} E^{i}_{q}(t) \beta_{i} + \sum_{i=1}^{m} F^{i}_{q}(t) \gamma_{i}.
\end{split}
\end{align}
Also, these functions \(A^{i}_{p}(t),\ldots,F_{q}^{i}(t)\) satisfy the properties
\begin{itemize}
\item[(1)] \(A_{p}^{i}(t)\to 0\) and \(C_{p}^{i}(t)\to 0\) as \(t\to 0\);
\item[(2)] \(B_{p}^{i}(t)\to b_{p}^{i}\) as \(t\to 0\);
\item[(3)] \(D_{q}^{i}(t)\to 0\) and \(E_{q}^{i}(t)\to 0\) as \(t\to 0\);
\item[(4)] \(F_{q}^{i}(t)\to \delta_{q}^{i}\), the Kronecker delta, as \(t\to 0\);
\end{itemize}

Let \(\mathfrak{h}\to \Delta^{\ast}\) via \(z\mapsto t=\exp(2\pi\sqrt{-1}z)\) be the 
universal cover.
It is not hard to see that
\begin{lemma}
\label{lem:class-h}
We have \(A_{p}^{i},C_{p}^{i},D_{q}^{i},E_{q}^{i}\in\mathbf{h}\). Moreover,
\(F_{q}^{i}\in\mathbf{h}\) if \(i\ne q\).
\end{lemma}

To obtain a frame of \(\mathcal{F}_{t}^{2}\) for \(|t|\ll 1\), we need to ``untwist''
the frame 
\(\{u_{p}(t),v_{r}(t)~|~p=1,\ldots,h+1~\mbox{and}~r=1,\ldots,m\}\)
and consider
the multi-valued sections
\begin{align}
\begin{split}
u'_{p}(t)&:=u_{p}(t)+zNu_{p}(t),\\
v'_{r}(t)&:=v_{r}(t)+zNv_{r}(t).
\end{split}
\end{align}
Then \(\{u'_{1}(t),\ldots,u'_{h+1}(t),
v'_{1}(t),\ldots,v'_{m}(t)\}\) is a (multi-valued) frame for \(\mathcal{F}^{2}_{t}\). Explicitly,
\begin{align}
\label{eq:basis-for-f2}
\begin{split}
u'_{p}(z)&=\sum_{i=1}^{m} \left(A^{i}_{p}(t)-\sum_{k=1}^{m}C_{p}^{k}(t)x_{ki}
+zC_{p}^{i}(t)\right) \alpha_{i}+
\sum_{i=1}^{2h+2} B^{i}_{p}(t) \beta_{i} + \sum_{i=1}^{m} C^{i}_{p}(t) \gamma_{i}\\
v'_{r}(z)&=\sum_{i=1}^{m} \left(D^{i}_{r}(t)-\sum_{k=1}^{m}F_{r}^{k}(t)x_{ki}
+zF_{r}^{i}(t)\right) \alpha_{i}+
\sum_{i=1}^{2h+2} E^{i}_{r}(t) \beta_{i} + \sum_{i=1}^{m} F^{i}_{r}(t) \gamma_{i}.
\end{split}
\end{align} 
We may and will further assume that 
\(z\) lies in a bounded vertical strip, i.e.~
\(\operatorname{Re}(z)\in[0,1]\).

Now we are ready to prove Theorem \ref{theorem:polarized-hs}.
\begin{proof}[Proof of Theorem \ref{theorem:polarized-hs}]
For CY \(\partial\bar{\partial}\)-threefolds, \(\mathrm{H}^{3,0}(X)\)
is always polarized by the sesquilinear 
pairing \(\tilde{Q}(\bullet,\overline{\bullet})\). Thus it
suffices to check
the second Hodge--Riemann bilinear relation on \(\mathrm{H}^{2,1}(X)\);
in other words, we have to check
\begin{eqnarray*}
\langle \sqrt{-1}u,\bar{u}\rangle:=-\sqrt{-1}\int_{X} u\cup \bar{u}\ge 0~\mbox{for}~
u\in\mathrm{H}^{2,1}(X)
\end{eqnarray*}
and the equality holds if and only if \(u=0\).

Given a degeneration \(\mathcal{X}\to\Delta\) as above, we constructed 
a basis for \(\mathcal{F}_{t}\) with \(t\ne 0\).
Thus it suffices to show that the hermitian matrix with respect to
the basis \(\{u'_{i}(z),v'_{j}(z)\}\) is positive definite, i.e.~the matrix
\begin{eqnarray}
\label{eq:matrix}
\begin{bmatrix}
M(u'_{p}(z),\bar{u}'_{q}(z)) & M(u'_{p}(z),\bar{v}'_{s}(z))\\
M(v'_{r}(z),\bar{u}'_{q}(z)) & M(v'_{r}(z),\bar{v}'_{s}(z))
\end{bmatrix}
\end{eqnarray}
is positive definite for \(\operatorname{Im}z\gg 0\). 
Here \(1\le p,q\le h+1\), \(1\le r,s\le m\), and
\begin{eqnarray*}
M(u,v):=\langle \sqrt{-1}u, v\rangle = -\sqrt{-1}\int_{X} u\cup v.
\end{eqnarray*}
Now under our hypothesis, we know that the limit 
\begin{eqnarray*}
a_{p,\bar{q}}^{\infty}:=\lim_{\operatorname{Im}z\to\infty}M(u'_{p}(z),\bar{u}'_{q}(z))
\end{eqnarray*}
exists for all \(1\le p,q\le h+1\). Moreover, 
the matrix \((a_{p,\bar{q}}^{\infty})_{1\le p,q\le h+1}\)
is positive definite since 
\(\mathrm{Gr}_{3}^{\mathcal{W}(M)}\mathrm{H}^{3}(X;\mathbb{C})\) carries
a polarized Hodge structure.
Consequently, the matrix
\begin{eqnarray*}
(a_{p\bar{q}}(z)=M(u'_{p}(z),\bar{u}'_{q}(z)))_{1\le p,q\le h+1}
\end{eqnarray*}
is positive definite when \(\operatorname{Im}(z)\gg 0\).
On the other hand, we have
\begin{eqnarray*}
Q(\alpha_{r},\alpha_{s})=Q(\alpha_{i},\beta_{j})=0~\mbox{for any}~r,s,i,j.
\end{eqnarray*}
Also, since \(\delta_{k}=\gamma_{k}-\sum_{l=1}^{m}x_{kl}\alpha_{l}\), from
\(Q(\delta_{r},N\delta_{s})=
Q(\delta_{r},\alpha_{s})=Q(\gamma_{r},\alpha_{s})\) and
\begin{eqnarray*}
Q(\gamma_{r},\alpha_{s})=Q(\gamma_{r},N\gamma_{s})=
-Q(N\gamma_{s},\gamma_{r})=Q(-N\gamma_{s},\gamma_{r})
=Q(\gamma_{s},N\gamma_{r})=Q(\gamma_{s},\alpha_{r}),
\end{eqnarray*}
we see that the real symmetric matrix \((q_{rs}=
Q(\delta_{r},\alpha_{s}))_{1\le r,s\le m}\) is 
positive definite, because the sesquilinear form
\(Q(u,N\bar{v})\) (where \(u,v\in 
\mathrm{Gr}_{4}^{\mathcal{W}(M)}\mathrm{H}^{3}(X;\mathbb{C})\))
is positive definite.

We compute
\begin{align*}
M(v'_{r}(z),\bar{v}'_{s}(z))&=2\operatorname{Im}(z) q_{rs} + P_{r\bar{s}} + H_{r\bar{s}}(z)
\end{align*}
for some constant \(P_{r\bar{s}}\) and some function \(H_{r\bar{s}}(z)\in\mathbf{h}\). 
It follows that the matrix
\begin{eqnarray*}
(d_{r\bar{s}} (z)=M(v'_{r}(z),\bar{v}'_{s}(z))=2\operatorname{Im}(z) 
(q_{rs}) + (P_{r\bar{s}}) + (H_{r\bar{s}}(z))
\end{eqnarray*}
is also positive definite when \(\operatorname{Im}(z)\gg 0\).

Put \(b_{p\bar{s}}(z)=M(u'_{p}(z),\bar{v}'_{s}(z))\)
and \(c_{r\bar{q}}(z)=M(v'_{r}(z),\bar{u}'_{q}(z))\). Then
\begin{eqnarray*}
\eqref{eq:matrix}=
\begin{bmatrix}
A(z) & B(z) \\
C(z) & D(z)
\end{bmatrix}\sim
\begin{bmatrix}
A(z)-C(z)D(z)^{-1}B(z) & \mathbf{0}_{(h+1)\times m} \\
\mathbf{0}_{m\times (h+1)} & D(z)
\end{bmatrix}.
\end{eqnarray*}
Let us investigate the inverse \(D(z)^{-1}\). Write
\begin{eqnarray*}
D(z) = 2\operatorname{Im}(z)\left(Q+\frac{P+H(z)}{2\operatorname{Im}(z)}\right)~
\mbox{with}~Q=(q_{r\bar{s}}).
\end{eqnarray*}
Since \(Q\) is invertible and \((2\operatorname{Im}(z))^{-1}(P+H(z))\to 0\)
as \(\operatorname{Im}(z)\to \infty\), we conclude that
\begin{eqnarray*}
D(z)^{-1}=\frac{1}{2\operatorname{Im}(z)}\left(Q^{-1}+G(z)\right)
\end{eqnarray*}
for some matrix \(G(z)\) with \(G(z)\to \mathbf{0}_{m\times m}\) as
\(\operatorname{Im}(z)\to\infty\).
Notice that the entries in \(C(z)\) and \(D(z)\) are all bounded.
We infer that
\begin{eqnarray*}
A(z)-C(z)D(z)^{-1}B(z)\to (a_{p\bar{q}}^{\infty})~\mbox{as}~\operatorname{Im}(z)\to\infty.
\end{eqnarray*}
In particular, \(A(z)-C(z)D(z)^{-1}B(z)\)
is positive definite for \(\operatorname{Im}(z)\gg 0\).
This shows that
\begin{eqnarray*}
\begin{bmatrix}
M(u'_{p}(z),\bar{u}'_{q}(z)) & M(u'_{p}(z),\bar{v}'_{s}(z))\\
M(v'_{r}(z),\bar{u}'_{q}(z)) & M(v'_{r}(z),\bar{v}'_{s}(z))
\end{bmatrix}
\end{eqnarray*}
is positive definite whenever \(\operatorname{Im}(z)\gg 0\). The proof is completed.
\end{proof}

\subsection{The \(\partial\bar{\partial}\)-lemma for finite distance degenerations}
We have constructed a basis for 
\(\mathcal{F}^{2}_{t}\) with \(t\ne 0\) in \eqref{eq:basis-for-f2}. 
It turns out that we are able to use this basis to prove that 
the \(\partial\bar{\partial}\)-lemma holds for
small smoothings.

Let \(f\colon\mathcal{X}\to\Delta\) be a semi-stable degeneration
and \(E:=f^{-1}(0)\) be the central fiber as before. 

\begin{proposition}
Assuming Hypothesis \ref{assumption:isomorphism}, then \(\mathcal{F}^{2}_{t}\cap 
\overline{\mathcal{F}^{2}_{t}}=(0)\) for \(t\) small.
\end{proposition}
\begin{proof}
To prove that \(\mathcal{F}^{2}_{t}\cap 
\overline{\mathcal{F}^{2}_{t}}=(0)\), it now suffices to show that 
\begin{equation}
\bigwedge_{i=1}^{h+1} u'_{i}(t)\wedge\overline{u'_{i}(t)}\wedge
\bigwedge_{i=1}^{m} v'_{i}(t)\wedge\overline{v'_{i}(t)}\ne 0
\end{equation}
where \(u'_{i}(t)\) and \(v'_{j}(t)\) are given in \eqref{eq:basis-for-f2}.
More accurately, using the fact \(N(\alpha_{i})=N(\beta_{i})=0\)
and \(N(\delta_{i})=N(\gamma_{i})=\alpha_{i}\), we get
\begin{align*}
\begin{split}
u'_{p}(t)&=\sum_{i=1}^{m} \left(A^{i}_{p}(t)-\sum_{k=1}^{m}C_{p}^{k}(t)x_{ki}
+zC_{p}^{i}(t)\right) \alpha_{i}+
\sum_{i=1}^{2h+2} B^{i}_{p}(t) \beta_{i} + \sum_{i=1}^{m} C^{i}_{p}(t) \gamma_{i}\\
v'_{q}(t)&=\sum_{i=1}^{m} \left(D^{i}_{q}(t)-\sum_{k=1}^{m}F_{q}^{k}(t)x_{ki}
+zF_{q}^{i}(t)\right) \alpha_{i}+
\sum_{i=1}^{2h+2} E^{i}_{q}(t) \beta_{i} + \sum_{i=1}^{m} F^{i}_{q}(t) \gamma_{i}.
\end{split}
\end{align*}

Using Lemma \ref{lem:class-h}, we conclude
\begin{equation}
\label{eq:dominant-term-in-middle}
\bigwedge_{i=1}^{h+1} u'_{i}(t)\wedge\overline{u'_{i}(t)}=
c\cdot\beta_{1}\wedge\cdots\wedge\beta_{2h+2} + \sum_{J} \phi^{J}(z)\omega_{J}.
\end{equation}
In the above displayed equation, \(c\) is a non-zero constant, \(\omega_{J}\)
is a \(2h+2\) form and \(\phi^{J}\in\mathbf{h}\).
The constant \(c\) is non-zero follows from the fact that 
\(\{u_{1},\ldots,u_{h+1}\}\) is a basis for \(I^{3,0}\oplus I^{2,1}\) and 
\(I^{0,3}=\overline{I^{3,0}}\) and \(I^{1,2}=\overline{I^{2,1}}\).

Observe that by the construction of our basis we have for each \(p\) and \(i\)
\begin{equation}
F_{q}^{i}(t)=\delta_{q}^{i}+G_{q}^{i}(t)
\end{equation}
for some holomorphic function \(G_{p}^{i}(t)\) with 
\(G_{p}^{i}\in\mathbf{h}\). 
Now we can re-write
\begin{align*}
v_{q}'(t)=(z-x_{qq})\alpha_{q}+\sum_{i=1}^{m}\phi^{i}_{q}(z)\alpha_{i}
+\sum_{i=1}^{2h+2} E^{i}_{q}(t) \beta_{i} + \sum_{i=1}^{m} 
\psi_{q}^{i}(z) \gamma_{i} + \gamma_{q}
\end{align*}
for some functions \(\phi_{q}^{i}(z),\psi_{q}^{i}(z)\in\mathbf{h}\). Therefore,
we have for each \(q\)
\begin{align}
\label{eq:dominant-2-2}
\begin{split}
v_{q}'(t)\wedge\overline{v_{q}'(t)} & = (z-\bar{z}-x_{qq}+\bar{x}_{qq})
\alpha_{q}\wedge\gamma_{q} + \sum_{J}\phi_{q}^{J}(z)\omega_{J}.
\end{split}
\end{align}
Here \(\omega_{J}\) are two forms other than \(\alpha_{q}\wedge\gamma_{q}\)
and \(\phi_{q}^{J}\in \mathbf{h}\).

Taking \eqref{eq:dominant-2-2} and \eqref{eq:dominant-term-in-middle}
into accounts, we see that
\begin{align}
\begin{split}
\bigwedge_{i=1}^{h+1} &u'_{i}(t)\wedge\overline{u'_{i}(t)}\wedge
\bigwedge_{i=1}^{m} v'_{i}(t)\wedge\overline{v'_{i}(t)}\\
&=\left(\prod_{i=1}^{m}(2\sqrt{-1}\operatorname{Im}(z)-2\sqrt{-1}
\operatorname{Im}(x_{ii}))+\phi(z)
\right)
\bigwedge_{i=1}^{m}\alpha_{i}\wedge\bigwedge_{i=1}^{2h+2}\beta_{i}\wedge
\bigwedge_{i=1}^{m}\gamma_{i}
\end{split}
\end{align}
with \(\phi\in\mathbf{h}\) and hence we conclude that it is non-zero as
long as \(\operatorname{Im}(z)\gg 0\). 
\end{proof}

\begin{corollary}
\label{cor:finite-distance-ddbar}
Let \(f\colon\mathcal{X}\to\Delta\) be as above.
We assume that
the general fiber \(X\) has 
complex dimension \(3\) and satisfies the following conditions.
\begin{itemize}
\item The central fiber of \(f\) is at a finite distance with respect to
the perid-map metric;
\item \(\mathrm{H}^{i}(X;\mathcal{O}_{X})=\mathrm{H}^{0}(X;\Omega_{X}^{j})=0\)
for \(1\le i,j\le 2\).
\end{itemize}
Then the \(\partial\bar{\partial}\)-lemma holds on \(X\).
\end{corollary}
\begin{proof}
This follows immediately from \cite{2019-Friedman-the-ddbar-lemma-for-general-clemens-manifolds}*{Corollary 1.6}. Indeed, the Hodge-de Rham spectral sequence degenerates at \(E_{1}\)
in the present situation (cf.~\cite{2008-Peters-Steenbrink-mixed-Hodge-structures}*{Corollary 11.24}).
\end{proof}

\section{Proof of the \texorpdfstring{\(\partial\bar{\partial}\)}{ddbar}-lemma
for Hashimoto--Sano's examples}

We can use the basis similar to \eqref{eq:basis-for-f2}
to prove that the
smooth non-K\"{a}hler Calabi--Yau threefolds constructed by Hashimoto and Sano also
support the \(\partial\bar{\partial}\)-lemma. 
We will explain the details in this section.

Let us review the construction of non-K\"{a}hler 
Calabi--Yau threefolds given in \cite{2023-Hashimoto-Sano-examples-of-non-kahler-calabi-yau-3-folds-with-arbitrarily-large-b2}
and record the data needed later.
Given a positive integer \(a>0\), 
Hashimoto and Sano constructed 
a smooth non-K\"{a}hler Calabi--Yau threefold \(X(a)\)
as a smoothing of a certain simple
normal crossing variety \(X_{0}(a)\). Let
us now review their construction and recall some necessary details. Fix
a positive integer \(a\).
Let \(X_{2}:=\mathbf{P}^{1}\times\mathbf{P}^{1}\times\mathbf{P}^{1}\) and
\(S\) be a very general hypersurface of degree \((2,2,2)\) in \(X_{2}\).
Then \(S\) is a smooth \(K3\) surface
with Picard number \(3\). Indeed,
\begin{equation}
\operatorname{Pic}(S)\cong \mathbb{Z}h_{1}\oplus\mathbb{Z}h_{2}\oplus\mathbb{Z}h_{3}
\end{equation}
where \(h_{i}\) is the pullback of the line bundle 
\(\mathcal{O}_{\mathbf{P}^{1}}(1)\) via
the projection to the \(i\textsuperscript{th}\) factor.

Let \(S\to\mathbf{P}^{1}\) be the projection to the first factor
and let \(C_{1},\ldots,C_{a}\) be distinct smooth fibers, 
i.e.~\(C_{k}\in |\mathcal{O}_{S}(1,0,0)|\). Finally, let
\(X_{1}\) be a blow-up of \(\mathbf{P}^{1}\times\mathbf{P}^{1}\times\mathbf{P}^{1}\)
along the curves \(C_{1},\ldots,C_{a}\) and \(C\), where 
\(C\) is a smooth curve in the linear system
\begin{equation}
|\mathcal{O}_{S}(16a^{2}-a+4,4-8a,4+8a)|.
\end{equation}
Note that the blow-up morphism 
\(\mu\colon X_{1}\to \mathbf{P}^1\times\mathbf{P}^{1}\times\mathbf{P}^{1}\)
induces an isomorphic between \(S\) and its proper transform,
which will be denoted by \(S_{1}\).

The projection \(\pi_{ij}
\colon S\to \mathbf{P}^{1}\times\mathbf{P}^{1}\) to the
\(i\textsuperscript{th}\) and \(j\textsuperscript{th}\) factors
is a double cover and induces an order two automorphism \(\iota_{ij}\)
on \(S\). Denote by \(\iota:=\iota_{12}\circ\iota_{13}\)
and \(\iota^{a}\) be its \(a\textsuperscript{th}\) power.
Put 
\begin{equation}
\iota_{a}:=\left.\iota^{a}\circ \mu\right|_{S_{1}}\colon S_{1}\to S.
\end{equation}
We can construct the pushout of \(X_{1}\) and \(X_{2}\) by 
the automorphism
\begin{equation}
\iota_{a}\colon S_{1}\to S
\end{equation}
where \(S_{1}\subset X_{1}\) and \(S\subset X_{2}\)
are viewed as \emph{subvarieties}. The action of \(\iota_{a}\)
on \(\operatorname{Pic}(S)\) is given by the matrix
(under the ordered basis \(\{h_{1},h_{2},h_{3}\}\))
\begin{equation}
\begin{bmatrix}
1 & 4a^{2}-2a & 2a^{2}+2a\\
0 & 1-2a & -2a\\
0 & 2a & 1+2a
\end{bmatrix}.
\end{equation}

The pushout is a SNC variety consisting 
of two components whose intersection is \(S\). Let us denote 
the pushout by \(X_{0}\). It is proven that \(X_{0}\) admits a 
semistable smoothing, that is, there exists a one
parameter family \(\mathcal{X}\to \Delta\) such that
\begin{itemize}
\item the central fiber is isomorphic to \(X_{0}\);
\item the total space \(\mathcal{X}\) is smooth;
\item the general fiber is smooth.
\end{itemize}

Let \(X\equiv X(a)\) be the smooth fiber. 
Using the semistable model, one can construct a limiting 
mixed Hodge structure on \(\mathrm{H}^{k}(X;\mathbb{C})\)
via the relative logarithmic de Rham complex. 
It is known that the Hodge--de Rham spectral sequence degenerates at \(E_{1}\) page
(cf.~\cite{2008-Peters-Steenbrink-mixed-Hodge-structures}).
In \cite{2023-Hashimoto-Sano-examples-of-non-kahler-calabi-yau-3-folds-with-arbitrarily-large-b2}, 
Hashimoto and Sano investigated 
the variety \(X\) in great detail and 
proved several results concerning \(X\). We summarize them as follows.

\begin{proposition}
\(X\) is simply-connected with \(b_{1}(X)=0\) and \(b_{2}(X)=a+3\). In addition, we have
the following vanishing
\begin{equation}
\mathrm{H}^{1}(X;\mathcal{O}_{X})=\mathrm{H}^{0}(X;\Omega_{X})
=\mathrm{H}^{2}(X;\mathcal{O}_{X})
=\mathrm{H}^{0}(X;\Omega_{X}^{2})=0.
\end{equation}
Furthermore, \(X\) is non-projective and in particular, \(X\) is non-K\"{a}hler.
\end{proposition}

The following lemma will be useful later.

\begin{lemma}
\label{lem:image-picard}
The image of the pullback
\begin{equation}
(\iota_{a}^{-1})^{\ast}\colon \mathrm{H}^{2}(X_{1};\mathbb{C})\to \mathrm{H}^{2}(S;\mathbb{C})
\end{equation}
is contained in \(\mathbb{C}h_{1}\oplus\mathbb{C}h_{2}
\oplus\mathbb{C}h_{3}\).
\end{lemma}
\begin{proof}
This follows from the construction.

Indeed, \(X_{1}\) is the blow-up of \(X=\mathbf{P}^{1}\times\mathbf{P}^{1}\times\mathbf{P}^{1}\)
along the curves \(C_{1},\ldots,C_{a},C\). Let 
\(Y_{k+1}:=\mathrm{Bl}_{C_{k}} Y_{k}\) with \(Y_{1}=X\)
and \(E_{k}\) be the exceptional divisor. Let 
\(T_{k+1}\) be the proper transform of \(T_{k}\) with \(T_{1}=S\).
We have
\begin{equation}
E_{k}\cap T_{k+1} \cong C_{k}\subset T_{k}
\end{equation}
under the blow-up morphism. Also in the last blow-up \(X_{1}=\mathrm{Bl}_{C}Y_{k+1}\),
we have
\begin{equation}
E\cap S_{1} \in |\mathcal{O}_{S_{1}}(D)|
\end{equation}
where \(D\) is the restriction of the divisor 
\((16a^{2}-a+4,4-8a,4+8a)\) under the morphism
\begin{equation}
S_{1}\xrightarrow{\cong}T_{a+1}\xrightarrow{\cong} T_{a}
\xrightarrow{\cong}\cdots \xrightarrow{\cong} T_{1}=S\subset X
\end{equation}
and \(S_{1}\subset X_{1}\) is the proper transform of \(T_{a+1}\).
The stated result follows since \(\mathrm{H}^{2}(X_{1};\mathbb{C})\)
is generated by the pullback of \(h_{1},h_{2},h_{3}\) and \(E_{1},\ldots,E_{a},E\).
\end{proof}

\subsection{The limiting mixed Hodge structure}
In this subsection, we will explicitly compute
Steenbrink's limiting mixed Hodge structure on the
middle cohomology of \(X\).

We have the spectral sequence induced by the ``monodromy weight filtration'' 
\(\mathcal{W}(M)\)
whose \(E_{1}\) page is given by
\begin{equation}
E_{1}^{-r,k+r} = \mathbf{H}^{k}
(\mathrm{Gr}_{r}^{\mathcal{W}(M)}\mathrm{Tot}(\mathcal{D}^{\bullet,\bullet}))
\Rightarrow \mathrm{Gr}_{k+r}^{\mathcal{W}(M)}\mathbf{H}^{k}(\mathrm{Tot}(\mathcal{D}^{\bullet,\bullet})).
\end{equation}
Moreover, the spectral sequence degenerates at \(E_{2}\) page.

Using the residue homomorphisms, one can compute the graded complexes
\begin{itemize}
\item \(\mathrm{Gr}_{-1}^{\mathcal{W}(M)}\mathrm{Tot}(\mathcal{D}^{\bullet,\bullet})
\simeq\mathbb{C}_{S}[-1]\);
\item \(\mathrm{Gr}_{0}^{\mathcal{W}(M)}\mathrm{Tot}(\mathcal{D}^{\bullet,\bullet})\simeq
\mathbb{C}_{X_{1}}\oplus\mathbb{C}_{X_{2}}\);
\item \(\mathrm{Gr}_{1}^{\mathcal{W}(M)}\mathrm{Tot}(\mathcal{D}^{\bullet,\bullet})
\simeq\mathbb{C}_{S}[-1]\)
\item \(\mathrm{Gr}_{r}^{\mathcal{W}(M)}\mathrm{Tot}(\mathcal{D}^{\bullet,\bullet})=0\)
for \(r\ne -1,0,1\).
\end{itemize}
Here the notation ``\(\simeq\)'' refers to ``is quasi-isomorphic to.''
Notice that the hypercohomology 
\(\mathbf{H}^{k}(\mathrm{Tot}(\mathcal{D}^{\bullet,\bullet}))\) 
of the total complex computes \(\mathrm{H}^{k}(X;\mathbb{C})\).

\begin{lemma}
\label{lem:monodromy-isomorphism}
We have
\begin{equation}
\label{eq:graded-4-3}
\mathrm{Gr}_{4}^{\mathcal{W}(M)}\mathrm{H}^{3}(X;\mathbb{C})
\cong \operatorname{ker}\left(\mathrm{H}^{2}(S;\mathbb{C})\to 
\mathrm{H}^{4}(X_{1};\mathbb{C})\oplus\mathrm{H}^{4}(X_{2};\mathbb{C})\right)
\end{equation}
and
\begin{equation}
\label{eq:graded-2-3}
\mathrm{Gr}_{2}^{\mathcal{W}(M)}\mathrm{H}^{3}(X;\mathbb{C})
\cong \operatorname{coker}\left(
\mathrm{H}^{2}(X_{1};\mathbb{C})\oplus\mathrm{H}^{2}(X_{2};\mathbb{C})
\to\mathrm{H}^{2}(S;\mathbb{C})\right)
\end{equation}
where the arrow in \eqref{eq:graded-4-3} is induced by Gysin maps and
the one in \eqref{eq:graded-2-3} is induced from the restrictions
\(\iota_{a}\colon S\to X_{1}\) and \(S\subset X_{2}\). Moreover, the
monodromy operator 
\begin{equation}
N\colon \mathrm{Gr}_{4}^{\mathcal{W}(M)}
\mathrm{H}^{3}(X;\mathbb{C})
\to \mathrm{Gr}_{2}^{\mathcal{W}(M)}
\mathrm{H}^{3}(X;\mathbb{C})
\end{equation}
is precisely the morphism induced from the identity on \(\mathrm{H}^{2}(S;\mathbb{C})\)
and it is an isomorphism in the present situation, i.e.~it satisfies
Hypothesis \ref{assumption:isomorphism}.
\end{lemma}
\begin{proof}
The \(E_{2}\)-term \(E_{2}^{-r,m+r}\) is the cohomology of 
\begin{equation}
E_{1}^{-r-1,m+r} \to E_{1}^{-r,m+r}\to E_{1}^{-r+1,m+r}.
\end{equation}
Put \(m=3\) and \(r=1\). We obtain the sequence
\begin{equation}
E_{1}^{-2,4}=0 \to E_{1}^{-1,4}=\mathrm{H}^{2}(S;\mathbb{C})\to E_{1}^{0,4}=
\mathrm{H}^{4}(X_{1};\mathbb{C})\oplus\mathrm{H}^{4}(X_{2};\mathbb{C}).
\end{equation}
This implies the first isomorphism.
Similarly, by putting \(m=3\) and \(r=-1\), we obtain
\begin{equation}
E_{1}^{0,2}=
\mathrm{H}^{2}(X_{1};\mathbb{C})\oplus\mathrm{H}^{2}(X_{2};\mathbb{C}) 
\to E_{1}^{1,2}=\mathrm{H}^{2}(S;\mathbb{C})\to E_{1}^{2,2}=0
\end{equation}
which implies the second statement.
Now since 
\begin{equation}
\mathrm{Gr}_{4}^{\mathcal{W}(M)}
\mathrm{H}^{3}(X;\mathbb{C})~
\mbox{and}~\mathrm{Gr}_{2}^{\mathcal{W}(M)}
\mathrm{H}^{2}(X;\mathbb{C})
\end{equation}
have the same dimension, it suffices to show that 
\(N\) is an injection. 

Look at the following diagram
\begin{equation}
\begin{tikzcd}
& &\mathrm{H}^{2}(S;\mathbb{C})\ar[r,"\varphi^{\vee}"]\ar[d,"\mathrm{id}"]
&\mathrm{H}^{4}(X_{1};\mathbb{C})\oplus\mathrm{H}^{4}(X_{2};\mathbb{C})\\
&\mathrm{H}^{2}(X_{1};\mathbb{C})\oplus\mathrm{H}^{2}(X_{2};\mathbb{C})\ar[r,"\varphi"]
&\mathrm{H}^{2}(S;\mathbb{C}) &
\end{tikzcd}
\end{equation}
where \(\varphi\) is the usual signed restriction
and \(\varphi^{\vee}\) is the signed Gysin map. It is now sufficient to prove that
if \(a\in\operatorname{ker}(\varphi^{\vee})\) and \(a\in\operatorname{im}(\varphi)\), 
then \(a=0\). Write \(a=\varphi(b)\) for some \(b\in 
\mathrm{H}^{2}(X_{1};\mathbb{C})\oplus\mathrm{H}^{2}(X_{2};\mathbb{C})\).
By Lemma \ref{lem:image-picard}, we may assume that \(b\in \mathrm{H}^{2}(X_{2};\mathbb{C})\).
The conditions imply that \(\varphi^{\vee}\varphi(b)=0\).
But we have
\begin{equation}
\varphi^{\vee}\varphi(b) = b\cup (2h_{1}+2h_{2}+2h_{3}),~b\in \mathrm{H}^{2}(X_{2};\mathbb{C})
\end{equation}
and one checks this is an injection (indeed, it is an isomorphism). Therefore, \(b=0\)
and hence \(a=0\) as desired.
\end{proof}

\subsection{A construction of the basis for the canonical extension}
Let the notations be as before. We construct a basis for 
\(F^{2}\) as follows. Recall that in the present situation we have
\begin{align}
W_{2}&=I^{2,0}\oplus I^{1,1}\oplus I^{0,2},\\
W_{3}&=W_{2}\oplus I^{3,0}\oplus I^{2,1}\oplus I^{1,2}\oplus I^{0,3},\\
W_{4}&=W_{3}\oplus I^{3,1}\oplus I^{2,2}\oplus I^{1,3}.
\end{align}
Note that from \eqref{eq:deligne-splitting},
both \(I^{2,0}\oplus I^{0,2}\) and \(I^{1,1}\) are
complex vector spaces with real structure. So is \(\bigoplus_{p=0}^{3} I^{p,3-p}\).
Since both \(X_{1}\) and \(X_{2}\) in Hashimoto--Sano's construction
are rational, we have \(\dim_{\mathbb{C}} I^{3,1}=1\) and \(I^{3,0}=0\).
Consequently, we can pick a \emph{real} basis
\(\{\epsilon_{1},\epsilon_{2}\}\) for \(I^{2,0}\oplus I^{0,2}\) and 
\(\{\alpha_{1},\ldots,\alpha_{m}\}\) 
for \(I^{1,1}\). Moreover, there are complex numbers \(w_{kl}\) with 
\(1\le k,l\le 2\) such that
\begin{equation}
\xi_{1}:=\sum_{l=1}^{2} w_{1l}\epsilon_{l}\in I^{2,0}~
\mbox{and}~\xi_{2}:=\sum_{l=1}^{2} w_{2l}\epsilon_{l}\in I^{0,2}.
\end{equation}
We may also assume that \(\bar{\xi}_{1}=\xi_{2}\), i.e.~
\begin{equation}
w_{11}=\bar{w}_{21}~\mbox{and}~w_{12}=\bar{w}_{22}.
\end{equation}

Since \(I^{2,1}\oplus I^{1,2}\) is also a 
complex vector space with real structure, we can pick 
a real basis \(\{\beta_{1},\ldots,\beta_{2h}\}\) for it.

Next let us consider \(I^{1,1}\oplus I^{2,2}\). 
This is again a complex vector spaces with real structure
by \eqref{eq:deligne-splitting} and therefore it also 
admits a real basis. We may extend the basis \(\{\alpha,\ldots,\alpha_{m}\}\)
to become a real basis 
\(\{\alpha_{1},\ldots,\alpha_{m},\gamma_{1},\ldots,\gamma_{m}\}\) 
for \(I^{1,1}\oplus I^{2,2}\) as before. Also, there are complex numbers
\(x_{kl}\in\mathbb{C}\) with \(1\le k,l\le m\) such that
\begin{equation}
\delta_{k}:=\gamma_{k}-\sum_{l=1}^{m} x_{kl}\alpha_{l}\in I^{2,2}.
\end{equation}
We may further assume that \(N(\delta_{k})=N(\gamma_{k})=\alpha_{k}\) because 
\(N\) is an isomorphism on the relevant grade piece by Lemma 
\ref{lem:monodromy-isomorphism}.

Finally, consider \(I^{3,1}\oplus I^{1,3}\oplus I^{2,0}\oplus I^{0,2}\).
Again by \eqref{eq:deligne-splitting}, this is a complex vector space
(of dimension four) with real structure. We may extend 
\(\{\epsilon_{1},\epsilon_{2}\}\) to a real basis 
\(\{\epsilon_{1},\epsilon_{2},\epsilon_{3},\epsilon_{4}\}\) for it.
We may assume that \(N(\epsilon_{3})=\epsilon_{1}\) and
\(N(\epsilon_{4})=\epsilon_{2}\).
Again there are complex numbers \(y_{kl}\)
with \(k=1,2\) and \(1\le l\le 4\) such that
\begin{align*}
\sigma_{1}:=\sum_{l=1}^{4}y_{1l}\epsilon_{l}\in I^{3,1}~\mbox{and}~
\sigma_{2}:=\sum_{l=1}^{4}y_{2l}\epsilon_{l}\in I^{1,3}
\end{align*}
Under our assumption, we have
\begin{equation}
N(\sigma_{1})=y_{13}\epsilon_{1}+y_{14}\epsilon_{2}\in I^{2,0}~\mbox{and}~
N(\sigma_{2})=y_{23}\epsilon_{1}+y_{24}\epsilon_{2}\in I^{0,2}.
\end{equation}
After rescaling, we may further assume \(y_{13}=w_{11}\), \(y_{14}=w_{12}\),
\(y_{23}=w_{21}\), and \(y_{24}=w_{22}\), i.e.~
\(N(\sigma_{1})=\xi_{1}\) and \(N(\sigma_{2})=\xi_{2}\).

\begin{figure}
\includegraphics[scale=0.55]{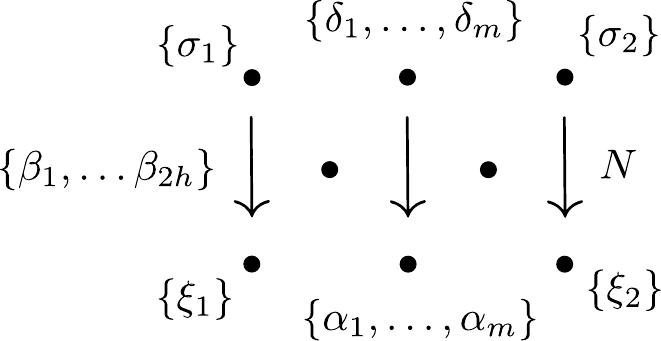}
\caption{The limiting mixed Hodge diamond for 
\(\mathrm{H}^{3}(X;\mathbb{C})\) and the bases for
the corresponding subspaces \(I^{p,q}\).}
\end{figure}

\subsection{Notation}
Let us fix the notation for the later use.
\begin{itemize}
\item Let \(\{u_{1},\ldots,u_{h}\}\) be
a basis for \(I^{2,1}\). We can
write \(u_{p}=\sum_{i=1}^{2h} b^{i}_{p} \beta_{i}\).
\item Let \(\{f_{1}:=\sigma_{1}\}\) be a basis for \(I^{3,1}\).
\item Let \(\{g_{1}:=\xi_{1}\}\) be a basis for \(I^{2,0}\).
\item Let \(\{v_{1}=\delta_{1},\ldots,v_{m}=\delta_{m}\}\) be a basis for \(I^{2,2}\). 
\end{itemize}
Since \(\mathcal{F}^{2}\) is a holomorphic subbundle in 
Deligne's canonical extension, we may extend
\begin{equation}
\{u_{1},\ldots,u_{h},v_{1},\ldots,v_{m},f_{1},g_{1}\}
\end{equation}
to become a local frame for \(\mathcal{F}^{2}\) over \(\Delta\). Explicitly, we can write
\begin{align*}
\begin{split}
u_{p}(t) &= \sum_{i=1}^{m} U^{\alpha_{i}}_{p}(t) \alpha_{i}+
\sum_{i=1}^{2h} U^{\beta_{i}}_{p}(t) \beta_{i} + \sum_{i=1}^{m} 
U^{\delta_{i}}_{p}(t) \delta_{i}+\sum_{i=1}^{2}
U^{\sigma_{i}}_{p}(t)\sigma_{i}+\sum_{i=1}^{2}U^{\xi_{i}}_{p}(t)\xi_{i},\\
v_{q}(t) &= \sum_{i=1}^{m} V^{\alpha_{i}}_{p}(t) \alpha_{i}+
\sum_{i=1}^{2h} V^{\beta_{i}}_{p}(t) \beta_{i} + \sum_{i=1}^{m} 
V^{\delta_{i}}_{p}(t) \delta_{i}+\sum_{i=1}^{2}
V^{\sigma_{i}}_{p}(t)\sigma_{i}+\sum_{i=1}^{2}V^{\xi_{i}}_{p}(t)\xi_{i},\\
f_{1}(t)&= \sum_{i=1}^{m} F^{\alpha_{i}}_{1}(t) \alpha_{i}+
\sum_{i=1}^{2h} F^{\beta_{i}}_{1}(t) \beta_{i} + \sum_{i=1}^{m} 
F^{\delta_{i}}_{1}(t) \delta_{i}+\sum_{i=1}^{2}
F^{\sigma_{i}}_{1}(t)\sigma_{i}+\sum_{i=1}^{2}F^{\xi_{i}}_{1}(t)\xi_{i},\\
g_{1}(t)&= \sum_{i=1}^{m} G^{\alpha_{i}}_{1}(t) \alpha_{i}+
\sum_{i=1}^{2h} G^{\beta_{i}}_{1}(t) \beta_{i} + \sum_{i=1}^{m} 
G^{\delta_{i}}_{1}(t) \delta_{i}+\sum_{i=1}^{2}
G^{\sigma_{i}}_{1}(t)\sigma_{i}+\sum_{i=1}^{2}G^{\xi_{i}}_{1}(t)\xi_{i}.\\
\end{split}
\end{align*}
In the above expression, the functions \(U,V,F,G\) are all \emph{holomorphic}.
In terms of the real basis
\(\{\alpha,\ldots,\alpha_{m},\beta_{1},\ldots,\beta_{2h},
\gamma_{1},\ldots,\gamma_{m},\epsilon_{1},\ldots,\epsilon_{4}\}\), 
we have
\begin{align*}
\begin{split}
u_{p}(t) &= \sum_{i=1}^{m} \left(U^{\alpha_{i}}_{p}(t)-
\sum_{k=1}^{m}U_{p}^{\delta_{k}}(t)x_{ki}\right) \alpha_{i}+
\sum_{i=1}^{2h} U^{\beta_{i}}_{p}(t) \beta_{i} + 
\sum_{i=1}^{m} U^{\delta_{i}}_{p}(t) \gamma_{i}\\
&+\sum_{j=1}^{2}\left(\sum_{i=1}^{2}y_{ij}U_{p}^{\sigma_{i}}(t)+
\sum_{i=1}^{2}w_{ij}U_{p}^{\xi_{i}}(t)\right)\epsilon_{j}
+\sum_{j=1}^{2}\left(\sum_{i=1}^{2}w_{ij}U_{p}^{\sigma_{i}}(t)\right)\epsilon_{j+2},\\
v_{p}(t) &= \sum_{i=1}^{m} \left(V^{\alpha_{i}}_{p}(t)-
\sum_{k=1}^{m}V_{p}^{\delta_{k}}(t)x_{ki}\right) \alpha_{i}+
\sum_{i=1}^{2h} V^{\beta_{i}}_{p}(t) \beta_{i} + 
\sum_{i=1}^{m} V^{\delta_{i}}_{p}(t) \gamma_{i}\\
&+\sum_{j=1}^{2}\left(\sum_{i=1}^{2}y_{ij}V_{p}^{\sigma_{i}}(t)+
\sum_{i=1}^{2}w_{ij}V_{p}^{\xi_{i}}(t)\right)\epsilon_{j}
+\sum_{j=1}^{2}\left(\sum_{i=1}^{2}w_{ij}V_{p}^{\sigma_{i}}(t)\right)\epsilon_{j+2},\\
f_{1}(t) &= \sum_{i=1}^{m} \left(F^{\alpha_{i}}_{1}(t)-
\sum_{k=1}^{m}F_{1}^{\delta_{k}}(t)x_{ki}\right) \alpha_{i}+
\sum_{i=1}^{2h} F^{\beta_{i}}_{1}(t) \beta_{i} + 
\sum_{i=1}^{m} F^{\delta_{i}}_{1}(t) \gamma_{i}\\
&+\sum_{j=1}^{2}\left(\sum_{i=1}^{2}y_{ij}F_{1}^{\sigma_{i}}(t)+
\sum_{i=1}^{2}w_{ij}F_{1}^{\xi_{i}}(t)\right)\epsilon_{j}
+\sum_{j=1}^{2}\left(\sum_{i=1}^{2}w_{ij}F_{1}^{\sigma_{i}}(t)\right)\epsilon_{j+2},\\
g_{1}(t) &= \sum_{i=1}^{m} \left(G^{\alpha_{i}}_{1}(t)-
\sum_{k=1}^{m}G_{1}^{\delta_{k}}(t)x_{ki}\right) \alpha_{i}+
\sum_{i=1}^{2h} G^{\beta_{i}}_{1}(t) \beta_{i} + 
\sum_{i=1}^{m} G^{\delta_{i}}_{1}(t) \gamma_{i}\\
&+\sum_{j=1}^{2}\left(\sum_{i=1}^{2}y_{ij}G_{1}^{\sigma_{i}}(t)+
\sum_{i=1}^{2}w_{ij}G_{1}^{\xi_{i}}(t)\right)\epsilon_{j}
+\sum_{j=1}^{2}\left(\sum_{i=1}^{2}w_{ij}G_{1}^{\sigma_{i}}(t)\right)\epsilon_{j+2}.
\end{split}
\end{align*}
Also, these functions satisfy the properties
\begin{itemize}
\item[(1)] \(U_{p}^{\alpha_{i}}, U_{p}^{\delta_{i}}, 
U_{p}^{\sigma_{i}}, U_{p}^{\xi_{i}}\in\mathbf{h}\)
and \(B_{p}^{\beta_{i}}(t)\to b_{p}^{i}\) as \(t\to 0\);
\item[(2)] \(V_{q}^{\alpha_{i}}, V_{q}^{\beta_{i}}, 
V_{q}^{\sigma_{i}}, V_{q}^{\xi_{i}}\in\mathbf{h}\)
and \(V_{q}^{\delta_{i}}(t)\to \delta_{q}^{i}\) as \(t\to 0\);
\item[(3)] \(F_{1}^{\alpha_{i}}, F_{1}^{\beta_{i}}, F_{1}^{\delta_{i}}, 
F_{1}^{\xi_{i}}\in\mathbf{h}\)
and \(F_{1}^{\sigma_{i}}(t)\to \delta_{1}^{i}\) as \(t\to 0\);
\item[(4)] \(G_{1}^{\alpha_{i}}, G_{1}^{\beta_{i}}, G_{1}^{\delta_{i}}, 
G_{1}^{\sigma_{i}}\in\mathbf{h}\)
and \(G_{1}^{\xi_{i}}(t)\to \delta_{1}^{i}\) as \(t\to 0\).
\end{itemize}

\subsection{Main calculations}
As before, let \(\mathfrak{h}\to \Delta^{\ast}\) be the 
universal cover.
In order to get a frame of \(\mathcal{F}_{t}^{2}\) for \(t\ne 0\), we ``untwist''
the frame \(\{u_{1}(t),\ldots,u_{h}(t),v_{1}(t),\ldots,v_{m}(t),f_{1}(t),g_{1}(t)\}\) 
and obtain the multi-valued sections
\begin{align}
\begin{split}
u'_{p}(t)&:=u_{p}(t)+zNu_{p}(t),\\
v'_{q}(t)&:=v_{q}(t)+zNv_{q}(t),\\
f'_{1}(t)&:=f_{1}(t)+zNf_{1}(t),\\
g'_{1}(t)&:=g_{1}(t)+zNg_{1}(t).
\end{split}
\end{align}
Then \(\{u'_{1}(t),\ldots,u'_{h}(t),
v'_{1}(t),\ldots,v'_{m}(t),f'_{1}(t),g_{1}'(t)\}\) 
is a frame for \(\mathcal{F}^{2}_{t}\)
for \(t\ne 0\). We may and will further assume that 
\(z\) lies in a bounded vertical strip, i.e.~
\(\operatorname{Re}(z)\in[0,1]\).

To prove that \(\mathcal{F}^{2}_{t}\cap 
\overline{\mathcal{F}^{2}_{t}}\), it now suffices to show that 
\begin{equation}
\bigwedge_{i=1}^{h} u'_{i}(t)\wedge\overline{u'_{i}(t)}\wedge
\bigwedge_{i=1}^{m} v'_{i}(t)\wedge\overline{v'_{i}(t)}
\wedge f'_{1}(t)\wedge\overline{f'_{1}(t)}\wedge
g'_{1}(t)\wedge\overline{g'_{1}(t)}\ne 0.
\end{equation}
More accurately, using the fact \(N(\alpha_{i})=N(\beta_{i})=N(\xi_{i})=0\),
\(N(\delta_{i})=N(\gamma_{i})=\alpha_{i}\), 
and \(N(\sigma_{i})=\xi_{i}\),
we get
\begin{align*}
\begin{split}
u'_{p}&(t) = \sum_{i=1}^{m} \left(U^{\alpha_{i}}_{p}(t)-
\sum_{k=1}^{m}U_{p}^{\delta_{k}}(t)x_{ki}+zU_{p}^{\delta_{i}}(t)\right) \alpha_{i}+
\sum_{i=1}^{2h} U^{\beta_{i}}_{p}(t) \beta_{i} + 
\sum_{i=1}^{m} U^{\delta_{i}}_{p}(t) \gamma_{i}\\
&+\sum_{j=1}^{2}\left(\sum_{i=1}^{2}y_{ij}U_{p}^{\sigma_{i}}(t)+
\sum_{i=1}^{2}w_{ij}\left(U_{p}^{\xi_{i}}(t)+zU_{p}^{\sigma_{i}}(t)\right)\right)\epsilon_{j}
+\sum_{j=1}^{2}\left(\sum_{i=1}^{2}w_{ij}U_{p}^{\sigma_{i}}(t)\right)\epsilon_{j+2},\\
v'_{p}&(t) = \sum_{i=1}^{m} \left(V^{\alpha_{i}}_{p}(t)-
\sum_{k=1}^{m}V_{p}^{\delta_{k}}(t)x_{ki}+zV_{p}^{\delta_{i}}(t)\right) \alpha_{i}+
\sum_{i=1}^{2h} V^{\beta_{i}}_{p}(t) \beta_{i} + 
\sum_{i=1}^{m} V^{\delta_{i}}_{p}(t) \gamma_{i}\\
&+\sum_{j=1}^{2}\left(\sum_{i=1}^{2}y_{ij}V_{p}^{\sigma_{i}}(t)+
\sum_{i=1}^{2}w_{ij}\left(V_{p}^{\xi_{i}}(t)+zV_{p}^{\sigma_{i}}(t)\right)\right)\epsilon_{j}
+\sum_{j=1}^{2}\left(\sum_{i=1}^{2}w_{ij}V_{p}^{\sigma_{i}}(t)\right)\epsilon_{j+2},\\
f'_{1}&(t)= \sum_{i=1}^{m} \left(F^{\alpha_{i}}_{1}(t)-
\sum_{k=1}^{m}F_{1}^{\delta_{k}}(t)x_{ki}+zF_{1}^{\delta_{i}}(t)\right) \alpha_{i}+
\sum_{i=1}^{2h} F^{\beta_{i}}_{1}(t) \beta_{i} + 
\sum_{i=1}^{m} F^{\delta_{i}}_{1}(t) \gamma_{i}\\
&+\sum_{j=1}^{2}\left(\sum_{i=1}^{2}y_{ij}F_{1}^{\sigma_{i}}(t)+
\sum_{i=1}^{2}w_{ij}\left(F_{1}^{\xi_{i}}(t)+zF_{1}^{\sigma_{i}}(t)\right)\right)\epsilon_{j}
+\sum_{j=1}^{2}\left(\sum_{i=1}^{2}w_{ij}F_{1}^{\sigma_{i}}(t)\right)\epsilon_{j+2},\\
g'_{1}&(t)= \sum_{i=1}^{m} \left(G^{\alpha_{i}}_{1}(t)-
\sum_{k=1}^{m}G_{1}^{\delta_{k}}(t)x_{ki}+zG_{1}^{\delta_{i}}(t)\right) \alpha_{i}+
\sum_{i=1}^{2h} G^{\beta_{i}}_{1}(t) \beta_{i} + 
\sum_{i=1}^{m} G^{\delta_{i}}_{1}(t) \gamma_{i}\\
&+\sum_{j=1}^{2}\left(\sum_{i=1}^{2}y_{ij}G_{1}^{\sigma_{i}}(t)+
\sum_{i=1}^{2}w_{ij}\left(G_{1}^{\xi_{i}}(t)+zG_{1}^{\sigma_{i}}(t)\right)\right)\epsilon_{j}
+\sum_{j=1}^{2}\left(\sum_{i=1}^{2}w_{ij}G_{1}^{\sigma_{i}}(t)\right)\epsilon_{j+2}.
\end{split}
\end{align*}

Using the items (1)--(4), we conclude that
\begin{equation}
\label{eq:dominant-term-in-middle-hs}
\bigwedge_{i=1}^{h} u'_{i}(t)\wedge\overline{u'_{i}(t)}=
c\cdot\beta_{1}\wedge\cdots\wedge\beta_{2h} + \sum_{J} \phi^{J}(z)\omega_{J}.
\end{equation}
In the displayed equation above, \(c\) is a non-zero constant, \(\omega_{J}\)
is a \(2h\) form and \(\phi^{J}\in\mathbf{h}\).
The constant \(c\) is non-zero follows from the fact that 
\(\{u_{1},\ldots,u_{h}\}\) is a basis for \(I^{2,1}\) and 
\(I^{1,2}=\overline{I^{2,1}}\).

Observe that by the construction of our basis we have for each \(p\) and \(i\)
\begin{equation}
V_{q}^{\delta_{i}}(t)=\delta_{q}^{i}+\tilde{V}_{q}^{i}(t)
\end{equation}
for some holomorphic function \(\tilde{V}_{q}^{i}(t)\) with 
\(\tilde{V}_{q}^{i}\in\mathbf{h}\). 
Now we can re-write
\begin{align*}
v_{q}'(t)=(z-x_{qq})\alpha_{q}+\sum_{i=1}^{m}\phi^{i}_{q}(z)\alpha_{i}
+\sum_{i=1}^{2h} V^{\beta_{i}}_{q}(t) \beta_{i} + \sum_{i=1}^{m} 
\psi_{q}^{i}(z) \gamma_{i} + \gamma_{q} + \sum_{i=1}^{4}\theta_{q}^{i}(z)\epsilon_{i}
\end{align*}
for some functions \(\phi_{q}^{i},\psi_{q}^{i},\theta_{q}^{i}\in\mathbf{h}\). Therefore,
we have for each \(q\)
\begin{align}
\label{eq:dominant-2-2-hs}
\begin{split}
v_{q}'(t)\wedge\overline{v_{q}'(t)} & = (z-\bar{z}-x_{qq}+\bar{x}_{qq})
\alpha_{q}\wedge\gamma_{q} + \sum_{J}\phi_{q}^{J}(z)\omega_{J}.
\end{split}
\end{align}
Here \(\omega_{J}\) are two forms other than \(\alpha_{q}\wedge\gamma_{q}\)
and \(\phi_{q}^{J}\in \mathbf{h}\).

Next, let us compute
\begin{equation}
f'_{1}(t)\wedge\overline{f'_{1}(t)}\wedge
g'_{1}(t)\wedge\overline{g'_{1}(t)}.
\end{equation}
Using the same trick, we conclude that
\begin{equation}
f'_{1}(t)\wedge\overline{f'_{1}(t)}\wedge
g'_{1}(t)\wedge\overline{g'_{1}(t)}=(w_{11}\bar{w}_{12}-w_{12}\bar{w}_{11})\epsilon_{1}
\wedge\cdots\wedge \epsilon_{4} + \sum_{I}\psi^{J}\omega_{J}
\end{equation}
where \(\omega_{J}\) is a four form and \(\psi^{J}\in\mathbf{h}\).

Finally, we observe that 
\begin{equation}
w_{11}\bar{w}_{12}-w_{12}\bar{w}_{11} = w_{11}w_{22} - w_{12}w_{21}\ne 0
\end{equation}
since \(\{\xi_{1},\xi_{2}\}\) is a basis for \(I^{2,0}\oplus I^{0,2}\).

Taking all the estimates above
into accounts, we see that
\begin{align*}
\begin{split}
\bigwedge_{i=1}^{h} &u'_{i}(t)\wedge\overline{u'_{i}(t)}\wedge
\bigwedge_{i=1}^{m} v'_{i}(t)\wedge\overline{v'_{i}(t)}
\wedge f'_{1}(t)\wedge\overline{f'_{1}(t)}\wedge
g'_{1}(t)\wedge\overline{g'_{1}(t)}\\
&=(w_{11}\bar{w}_{12}-w_{12}\bar{w}_{11})\left(\prod_{i=1}^{m}(2\sqrt{-1}\operatorname{Im}(z)-2\sqrt{-1}
\operatorname{Im}(x_{ii}))+\phi(z)
\right)\Omega
\end{split}
\end{align*}
where
\begin{equation*}
\Omega = \bigwedge_{i=1}^{4}\epsilon_{i}\wedge 
\bigwedge_{i=1}^{m}\alpha_{i}\wedge\bigwedge_{i=1}^{2h}\beta_{i}\wedge
\bigwedge_{i=1}^{m}\gamma_{i}
\end{equation*}
and \(\phi\in\mathbf{h}\). Therefore we conclude that it is non-zero as
long as \(\operatorname{Im}(z)\gg 0\). 
\begin{theorem}
\label{thm:hs-example-ddbar}
Let notations be as above.
For \(t\in\Delta^{\ast}\) with \(|t|\ll 1\), we have
\begin{equation}
\mathcal{F}^{2}_{t}\cap \overline{\mathcal{F}^{2}_{t}}=(0).
\end{equation}
Hence the \(\partial\bar{\partial}\)-lemma holds on 
the non-K\"{a}hler Calabi--Yau threefolds constructed by Hashimoto and Sano
in \cite{2023-Hashimoto-Sano-examples-of-non-kahler-calabi-yau-3-folds-with-arbitrarily-large-b2}.
\end{theorem}
\begin{proof}
The first statement \(\mathcal{F}^{2}_{t}\cap \overline{\mathcal{F}}^{2}_{t}=(0)\)
follows from our previous calculation, while the second one follows from
\cite{2019-Friedman-the-ddbar-lemma-for-general-clemens-manifolds}*{Corollary 1.6}.
\end{proof}

\begin{theorem}
\label{thm:hs-example-polarized}
The Hodge structure on \(\mathrm{H}^{3}(X;\mathbb{C})\)
is polarized by \(Q(C-,-)\).
\end{theorem}
\begin{proof}
This follows from the argument of Theorem \ref{theorem:polarized-hs}.
From the formula
\eqref{eq:pairing-gr33} and \eqref{eq:pairing-gr24} which is true without any assumption
on the shape of the limiting Hodge diamond,
the second Hodge--Riemann bilinear relation still holds on \(\mathrm{H}^{2,1}(X)\).
\end{proof}





\begin{bibdiv}
\begin{biblist}

\bib{1990-Candelas-Green-Hubsch-rolling-among-calabi-yau-vacua}{article}{
      author={Candelas, P.},
      author={Green, P.},
      author={{H\"{u}bsch}, Tristan},
       title={Rolling among {C}alabi--{Y}au vacua},
        date={1990},
     journal={Nuclear Physics B},
      volume={330},
      number={1},
       pages={49\ndash 102},
}

\bib{2023-Collins-Gukov-Picard-Yau-special-lagrangian-cycles-and-calabi-yau-transitions}{article}{
      author={Collins, Tristan~C.},
      author={Gukov, Sergei},
      author={Picard, Sebastien},
      author={Yau, Shing-Tung},
       title={Special {L}agrangian cycles and {C}alabi--{Y}au transitions},
        date={2023},
        ISSN={0010-3616},
     journal={Comm. Math. Phys.},
      volume={401},
      number={1},
       pages={769\ndash 802},
         url={https://doi.org/10.1007/s00220-023-04655-3},
      review={\MR{4604907}},
}

\bib{2024-Chen}{article}{
      author={Chen, Kuan-Wen},
       title={On the {H}odge structure of global smoothings of normal crossing
  varieties},
        date={2024},
     journal={in preparation},
}

\bib{1986-Cattani-Kaplan-Schmid-degeneration-of-hodge-structures}{article}{
      author={Cattani, E.},
      author={Kaplan, A.},
      author={Schmid, W.},
       title={{Degeneration of Hodge structures}},
        date={1986},
     journal={Ann. of Math.},
      volume={123},
       pages={457\ndash 535},
}

\bib{2024-Collins-Picard-Yau-stability-of-the-tangent-bundle-through-conifold-transitions}{article}{
      author={Collins, Tristan},
      author={Picard, Sebastien},
      author={Yau, Shing-Tung},
       title={Stability of the tangent bundle through conifold transitions},
        date={2024},
        ISSN={0010-3640},
     journal={Comm. Pure Appl. Math.},
      volume={77},
      number={1},
       pages={284\ndash 371},
         url={https://doi.org/10.1002/cpa.22135},
      review={\MR{4666627}},
}

\bib{1971-Deligne-theorie-de-Hodge-II}{article}{
      author={Deligne, Pierre},
       title={Th\'{e}orie de {H}odge. {II}},
        date={1971},
        ISSN={0073-8301},
     journal={Institut des Hautes \'{E}tudes Scientifiques. Publications
  Math\'{e}matiques},
      number={40},
       pages={5\ndash 57},
      review={\MR{498551}},
}

\bib{1977-Deligne-Griffiths-Morgan-Sullivan-real-homotopy-theory-of-kahler-manifolds}{article}{
      author={Deligne, P.},
      author={Griffiths, F.},
      author={Morgan, J.},
      author={Sullivan, D.},
       title={Real homotopy theory of {K}\"{a}hler manifolds},
        date={1977},
        ISSN={0042-1316},
     journal={Akademiya Nauk SSSR i Moskovskoe Matematicheskoe Obshchestvo.
  Uspekhi Matematicheskikh Nauk},
      volume={32},
      number={3(195)},
       pages={119\ndash 152, 247},
        note={Translated from the English by Ju. I. Manin},
      review={\MR{0460700}},
}

\bib{2012-Fu-Li-Yau-balanced-metrics-on-non-kahler-calabi-yau-threefolds}{article}{
      author={Fu, Jixiang},
      author={Li, Jun},
      author={Yau, Shing-Tung},
       title={Balanced metrics on non-{K}\"{a}hler {C}alabi--{Y}au threefolds},
        date={2012},
        ISSN={0022-040X},
     journal={Journal of Differential Geometry},
      volume={90},
      number={1},
       pages={81\ndash 129},
      review={\MR{2891478}},
}

\bib{2019-Friedman-the-ddbar-lemma-for-general-clemens-manifolds}{article}{
      author={Friedman, Robert},
       title={The {$\partial\overline\partial$}-lemma for general {C}lemens
  manifolds},
        date={2019},
        ISSN={1558-8599},
     journal={Pure Appl. Math. Q.},
      volume={15},
      number={4},
       pages={1001\ndash 1028},
      review={\MR{4085665}},
}

\bib{1986-Friedman-simultaneous-resolution-of-threefold-double-points}{article}{
      author={Friedman, Robert},
       title={Simultaneous resolution of threefold double points},
        date={1986},
        ISSN={0025-5831},
     journal={Math. Ann.},
      volume={274},
      number={4},
       pages={671\ndash 689},
      review={\MR{848512}},
}

\bib{2014-Fujisawa-polarization-on-limiting-mixed-hodge-structures}{article}{
      author={Fujisawa, Taro},
       title={Polarizations on limiting mixed {H}odge structures},
        date={2014},
     journal={Journal of Singularities},
      volume={8},
       pages={146\ndash 193},
      review={\MR{3395244}},
}

\bib{2023-Hashimoto-Sano-examples-of-non-kahler-calabi-yau-3-folds-with-arbitrarily-large-b2}{article}{
      author={Hashimoto, Kenji},
      author={Sano, Taro},
       title={Examples of non-{K}\"{a}hler {C}alabi--{Y}au 3-folds with
  arbitrarily large {$b_2$}},
        date={2023},
        ISSN={1465-3060},
     journal={Geom. Topol.},
      volume={27},
      number={1},
       pages={131\ndash 152},
  url={https://doi-org.ezp-prod1.hul.harvard.edu/10.2140/gt.2023.27.131},
      review={\MR{4584262}},
}

\bib{2018-Lee-a-Hodge-theoretic-criterion-for-finite-WP-degenerations-over-a-higher-dimensional-base}{article}{
      author={Lee, Tsung-Ju},
       title={A {H}odge theoretic criterion for finite {W}eil--{P}etersson
  degenerations over a higher dimensional base},
        date={2018},
        ISSN={1073-2780},
     journal={Mathematical Research Letters},
      volume={25},
      number={2},
       pages={617\ndash 647},
      review={\MR{3826838}},
}

\bib{2022-Li-polarized-hodge-structures-for-clemens-manifolds}{article}{
      author={Li, Chi},
       title={Polarized {H}odge structures for {C}lemens manifolds},
        date={202202},
      eprint={2202.10353},
         url={https://arxiv.org/pdf/2202.10353.pdf},
}

\bib{2018-Lee-Lin-Wang-towards-A+B-theory-in-conifold-transitions-for-CY-threefolds}{article}{
      author={Lee, Yuan-Pin},
      author={Lin, Hui-Wen},
      author={Wang, Chin-Lung},
       title={Towards {$A+B$} theory in conifold transitions for
  {C}alabi--{Y}au threefolds},
        date={2018},
        ISSN={0022-040X},
      volume={110},
      number={3},
       pages={495\ndash 541},
      review={\MR{3880232}},
}

\bib{2019-Popovici-holomorphic-deformations-of-balanced-calabi-yau-d-dbar-manifolds}{article}{
      author={Popovici, Dan},
       title={Holomorphic deformations of balanced {C}alabi--{Y}au
  {$\partial\overline\partial$}-manifolds},
        date={2019},
        ISSN={0373-0956},
     journal={Ann. Inst. Fourier (Grenoble)},
      volume={69},
      number={2},
       pages={673\ndash 728},
         url={https://mathscinet.ams.org/mathscinet-getitem?mr=3978322},
      review={\MR{3978322}},
}

\bib{2008-Peters-Steenbrink-mixed-Hodge-structures}{book}{
      author={Peters, Chris A.~M.},
      author={Steenbrink, Joseph H.~M.},
       title={Mixed {H}odge structures},
      series={Ergebnisse der Mathematik und ihrer Grenzgebiete. 3. Folge. A
  Series of Modern Surveys in Mathematics [Results in Mathematics and Related
  Areas. 3rd Series. A Series of Modern Surveys in Mathematics]},
   publisher={Springer-Verlag, Berlin},
        date={2008},
      volume={52},
        ISBN={978-3-540-77015-2},
      review={\MR{2393625}},
}

\bib{1987-Reid-the-moduli-space-of-3-folds-with-k-0-may-nevertheless-be-irreducible}{article}{
      author={Reid, Miles},
       title={The moduli space of {$3$}-folds with {$K=0$} may nevertheless be
  irreducible},
        date={1987},
        ISSN={0025-5831},
     journal={Mathematische Annalen},
      volume={278},
      number={1-4},
       pages={329\ndash 334},
      review={\MR{909231}},
}

\bib{1975-Steenbrink-limits-of-hodge-structures}{article}{
      author={Steenbrink, Joseph},
       title={Limits of {H}odge structures},
        date={1975/76},
        ISSN={0020-9910},
     journal={Invent. Math.},
      volume={31},
      number={3},
       pages={229\ndash 257},
      review={\MR{429885}},
}

\bib{2003-Wang-quasi-hodge-metrics-and-canonical-singularities}{article}{
      author={Wang, Chin-Lung},
       title={Quasi-{H}odge metrics and canonical singularities},
        date={2003},
        ISSN={1073-2780},
     journal={Math. Res. Lett.},
      volume={10},
      number={1},
       pages={57\ndash 70},
      review={\MR{1960124}},
}

\bib{1997-Wang-on-the-incompleteness-of-the-weil-petersson-metric-along-degenerations-of-calabi-yau-manifolds}{article}{
      author={Wang, Chin-Lung},
       title={On the incompleteness of the {W}eil--{P}etersson metric along
  degenerations of {C}alabi--{Y}au manifolds},
        date={1997},
        ISSN={1073-2780},
     journal={Math. Res. Lett.},
      volume={4},
      number={1},
       pages={157\ndash 171},
      review={\MR{1432818}},
}

\bib{1978-Yau-on-the-ricci-curvature-of-a-compact-kahler-manifold-and-the-complex-monge-ampere-equation-i}{article}{
      author={Yau, Shing~Tung},
       title={On the {R}icci curvature of a compact {K}\"{a}hler manifold and
  the complex {M}onge--{A}mp\`ere equation. {I}},
        date={1978},
        ISSN={0010-3640},
     journal={Commun. Pure Appl. Math.},
      volume={31},
      number={3},
       pages={339\ndash 411},
      review={\MR{480350}},
}

\end{biblist}
\end{bibdiv}
\end{document}